\documentclass[bj]{imsart}

\usepackage[utf8]{inputenc}

\usepackage{amsthm,amsmath,amsfonts,amssymb}
\usepackage{dsfont}
\usepackage{bm,bbm}
    
\usepackage{geometry}
\usepackage{nicefrac}
\usepackage[utf8]{inputenc}
\usepackage[T1]{fontenc}    % use 8-bit T1 fonts
\usepackage{multirow}
\usepackage{booktabs}
\usepackage[table,xcdraw]{xcolor}
\usepackage{soul}

\usepackage[colorlinks, citecolor=blue, urlcolor=blue]{hyperref}

\usepackage{caption}
\usepackage{subcaption}

\usepackage{enumitem}

\renewcommand{\theenumi}{{\rm(\roman{enumi})}}

\usepackage[textsize=footnotesize]{todonotes}
\usepackage{xcolor}

\startlocaldefs
\theoremstyle{plain}

\newtheorem{theorem}{Theorem}[section]
\newtheorem{lemma}[theorem]{Lemma}
\newtheorem{proposition}[theorem]{Proposition}
\theoremstyle{remark}
\newtheorem{definition}[theorem]{Definition}
\newtheorem{example}{Example}

\newtheorem{remark}[theorem]{Remark}
\endlocaldefs

\newcommand{\todoML}[2][]{\todo[color=green!40,size=\tiny,#1]{ML: #2}}

\usepackage{natbib}
\makeatletter
\renewcommand*\env@matrix[1][\arraystretch]{%
  \edef\arraystretch{#1}%
  \hskip -\arraycolsep
  \let\@ifnextchar\new@ifnextchar
  \array{*\c@MaxMatrixCols c}}
\makeatother

\newcommand{\E}{{\mathbb E}}
\newcommand{\Q}{{\mathbb Q}}
\newcommand{\R}{{\mathbb R}}
\newcommand{\N}{{\mathbb N}}
\newcommand{\V}{{\mathbb V}}

\newcommand{\Fcal}{{\mathcal F}}
\newcommand{\Gcal}{{\mathcal G}}
\newcommand{\Hcal}{{\mathcal H}}

\newcommand{\Mcal}{{\mathcal M}}
\newcommand{\Ncal}{{\mathcal N}}

\newcommand{\Xcal}{{\mathcal X}}

\DeclareMathOperator*{\argmin}{argmin}
\DeclareMathOperator*{\argmax}{argmax}

\DeclareMathOperator*{\maximize}{maximize}

\DeclareMathOperator{\dotp}{{\bm\cdot}}
\DeclareMathOperator{\VaR}{{VaR}}
\DeclareMathOperator{\RVaR}{{RVaR}}

\usepackage{tikz}

\newcommand{\reg}{\text{Regret}}

\DeclareMathOperator{\diam}{diam}

\newcommand{\Hnull}{\Hcal_0}
\newcommand{\rbox}[1]{{\color{red} \fbox{#1}}}

\newcommand{\umax}{u_{\max}}
\newcommand{\Abd}{A^{\mathrm{bd}}}
\newcommand{\Apsi}{A^{\psi}}

\begin{document}

\begin{frontmatter}
\title{Anytime-valid sequential testing for elicitable functionals via supermartingales}
%\title{A sample article title with some additional note\thanksref{t1}}
\runtitle{Anytime-valid sequential testing for elicitable functionals via supermartingales}
%\thankstext{T1}{A sample additional note to the title.}

\begin{aug}
\author[A]{\fnms{Philippe} \snm{Casgrain}\ead[label=e1]{philippe.casgrain@math.ethz.ch}},
\author[B]{\fnms{Martin} \snm{Larsson}\ead[label=e2]{larsson@cmu.edu}}
\and
\author[C]{
\fnms{Johanna} \snm{Ziegel}\ead[label=e3]{johanna.ziegel@stat.unibe.ch}
}
%%%%%%%%%%%%%%%%%%%%%%%%%%%%%%%%%%%%%%%%%%%%%%
%% Addresses                                %%
%%%%%%%%%%%%%%%%%%%%%%%%%%%%%%%%%%%%%%%%%%%%%%
\address[A]{ETH Zürich, 
\printead{e1}}

\address[B]{Carnegie Mellon University, 
\printead{e2}}

\address[C]{University of Bern, 
\printead{e3}}

\end{aug}

\begin{abstract}
We design sequential tests for a large class of nonparametric null hypotheses based on elicitable and identifiable functionals. Such functionals are defined in terms of scoring functions and identification functions, which are ideal building blocks for constructing nonnegative supermartingales under the null. This in turn yields anytime valid tests via Ville's inequality. Using regret bounds from Online Convex Optimization, we obtain rigorous guarantees on the asymptotic power of the tests for a wide range of alternative hypotheses. Our results allow for bounded and unbounded data distributions, assuming that a sub-$\psi$ tail bound is satisfied.
%\texttt{<Abstract here>}
\end{abstract}

% \begin{keyword}
% \kwd{First keyword}
% \kwd{second keyword}
% \end{keyword}

\end{frontmatter}

% \tableofcontents

% \newpage

\section{Introduction}

We design sequential tests and confidence sequences for a large class of nonparametric null hypotheses based on elicitable and identifiable functionals. Such functionals include moments, quantiles, expectiles, and many other examples, all of which can be tested using the approach developed here. The null hypotheses that we cover are highly composite and nonparametric; for instance, the null could consist of \emph{all} distributions whose median, say, is a given value. Our tests are sequential, or \emph{anytime valid}, in the sense that data is observed sequentially through time, and at each point in time the decision to stop or continue may depend on all available data without compromising Type-I error guarantees. We also obtain guarantees on the power of our tests with respect to large composite nonparametric alternative hypotheses. %These alternative hypotheses are determined from observed data in an online fashion in a manner designed to reject the null hypothesis quickly, if indeed the null is false.

The basic mechanism we use to construct anytime valid tests rests on the notion of \emph{test (super-)martingale} due to \cite{MR2849911}. The idea is simple but powerful: a test statistic that is a nonnegative supermartingale if the null hypothesis is true can only reach large values with small probability. This can be quantified using Ville's inequality. Thus if one rejects the null only when a sufficiently large value of the test statistic has been observed, Type-I error control is ensured.

The first contribution of our work is the observation that elicitable and identifiable functionals in the sense of \cite{LambertPennockETAL2008,gneiting2011making,FisslerZiegel2016} are ideal building blocks for constructing test supermartingales. Combined with a construction known as predictable mixing, one immediately obtains large families of test supermartingales which can be used as possible test statistics for the null hypothesis defined by a particular elicitable or identifiable functional.

The predictable mixing construction can be interpreted in terms of betting or trading. Finding a useful predictable mixture corresponds to determining a profitable trading strategy. The supermartingale condition under the null ensures that profits are limited if the null is true. But if the null is false, it may be possible to ``bet against the null'' in a way that leads to large profits and, hence, reject the null. Doing so requires two things. First, in order to bet against the null, one must specify a suitable distribution to bet on. Second, given this distribution, one must find a strategy that is likely to be profitable.% if the alternative is true.

The second contribution of our work is to address these two points at once by making use of ideas from online convex optimization (OCO). We demonstrate how off-the-shelf algorithms can be used to produce strong trading strategies. This leads to powerful test supermartingales given as predictable mixtures of the basic set of test supermartingales constructed from the elicitable or identifiable functional used to specify the null hypothesis. A major advantage of this approach is that these algorithms come with performance guarantees in the form of regret bounds. These regret bounds translate into rigorous guarantees on the power of the resulting sequential test under a wide variety of alternative hypotheses.

%The third contribution of our work is that we develop ways of allowing unbounded data distributions. [but what about Howard et al., etc...?]

Sequential testing goes back to \cite{MR13275}. A large body of literature on the subject exists, and martingale techniques have played an important role from the beginning; see Appendix~D in \cite{wau_ram_20} for an overview. The concept of a test martingale was introduced in \cite{MR2849911}, and there has recently been a number of papers related to this circle of ideas, for instance \cite{howard2018timeuniform,howard2020timeuniform} which derive time-uniform confidence sequences and concentration bounds. In particular, the closely related notion of e-variables and e-processes have received significant attention; see e.g.\ \cite{gru_hei_koo_19,10.1214/20-AOS2020,xu_wang_ramdas_21,ram_ruf_lar_koo_20,MR4364897} as well as Remark~\ref{R_e_process} below. Most closely related to our paper is the work of \cite{wau_ram_20}, which develops anytime valid confidence sequences for the mean of a sequence of bounded random variables. That paper makes use of the same betting perspective, which enables the authors to obtain powerful confidence sequences. It also discusses various related strands of literature and the history of the subject. However, the authors do not consider other functionals beyond the mean, and they rely on the boundedness of the data in an essential way. Our work generalizes both of these points. Moreover, they do not make use of OCO to obtain regret bounds which then translate into statements about power. Our use of regret bounds is reminiscent of \cite{MR4364897}, where regret bounds are used to derive power guarantees for the particular problem of testing exchangeability of binary sequences. Another paper related to ours is \cite{10.1093/biomet/asab047}, which treats the problem of probability forecasting. The authors rely on the betting analogy to construct sequential tests for the statistical significance of score differences of competing forecasts.

The concepts of elicitability and identifiability go back to the PhD thesis of \cite{Osband1985}. However, the term elicitability was coined later \citep{LambertPennockETAL2008}, and it was popularized by \cite{gneiting2011making,steinwart2014elicitation,frongillo2015elicitation, FisslerZiegel2016}. However, prior to our work, no systematic approach to the sequential testing problem for elicitable and identifiable functionals had been developed. Thus our paper contributes to this literature as well, the essential point being the link between test supermartingales and the concepts of elicitability and identifiability.

The  paper is organized as follows. In Section~\ref{sec:seq_test} we review the definition of test supermartingales and how they can be used to construct powerful anytime valid tests via predictable mixing. In Section~\ref{sec_null_hypothesis} we discuss elicitability and identifiability, both as a way of specifying nonparametric null hypotheses and as the basis for constructing test supermartingales. Importantly, we show how sub-$\psi$ tail bounds can be leveraged to handle unbounded data. Section~\ref{sec:power-via-OCO} discusses how regret bounds from online convex optimization lead to statements about asymptotic power of the tests. In Section~\ref{sec:confidence-sequences} we briefly discuss the related issue of confidence sequences. Section~\ref{sec:simulation_study} contains a simulation study illustrating the techniques developed in this paper. %Section~\ref{sec:conclusion} concludes.

\section{Sequential testing via supermartingales} \label{sec:seq_test} %Sequential Testing and Anytime Validity}

We consider a sequential testing environment in which a discrete-time stochastic process $(X_t)_{t \in \N}$, taking values in some measurable space $\Xcal$, is observed sequentially through time. The process $X$ is called the \emph{data generating process}. Concrete examples of data generating processes include patient data collected from clinical trials or daily profit and loss values of a trading strategy. The filtration generated by the data generating process is denoted $\Fcal := (\Fcal_t)_{t\in\N}$ where $\Fcal_t := \sigma( X_1,\ldots,X_t )$ is the information set generated by the data collected until time $t$. We let $\Fcal_0$ denote the trivial $\sigma$-algebra.

A statistical hypothesis is a collection $\Hcal \subseteq \Mcal_1(\Xcal^\N)$, where $\Mcal_1(\Xcal^\N)$ is the set of all probability distributions of the data generating process. Thus an element $P \in \Mcal_1(\Xcal^\N)$ is a distribution \emph{over the entire sequence} $(X_t)_{t \in \N}$. The hypothesis $\Hcal$ encodes the belief that the realized data was governed by one of the distributions $P \in \Hcal$. A \emph{(sequential) test} for a given null hypothesis $\Hnull$ is defined as an $\Fcal$-stopping time $\tau$ that specifies the time at which $\Hcal_0$ is rejected. The requirement that $\tau$ be a stopping time means that the decision whether to stop and reject $\Hcal_0$, or to continue and observe more data, is only based on data available at the time the decision is made. Stopping times are allowed to take the value infinity, and this corresponds to the possibility that the test never rejects the null.

%A hypothesis $\Hcal$ encodes beliefs about which probability measures may govern the data generating process $X$. Formally, we define a test for a null hypothesis $\Hcal$ as a (possibly infinite) $\Fcal$-stopping-time $\tau$ which specifies, as a function of the data, the rule by which we reject $\Hcal$. The null hypothesis, written as $\Hnull$, typically denotes the hypothesis that a test aims to reject.

\subsection{Anytime validity and test supermartingales}

%Given a data generating process $X$, we are interested in testing a hypothesis $\Hnull$ based on the data $\{x_i\}_i$. We allow ourselves to make decisions to stop or continue collecting samples and to re-test the hypothesis at any sequence of $\Fcal$-adapted times.
In contrast to traditional hypothesis testing with a fixed and known total sample size, the total sample size that will be produced by the $X_t$ before stopping is not known in advance. As a consequence, repeatedly evaluating a test designed for fixed finite sample sizes will generate an inflated Type-I error (or size) of the test, see for example \cite{albers2019problem,o1971present}. Although techniques such as multiple comparison p-value adjustments, see for example \cite{hsu1996multiple}, exist to correctly modify tests a-posteriori, these methods have quickly decaying power as the number of repeated tests grows large and require that the number of tests to be performed be known in advance. In order to avoid these issues, we work with the concept of an \emph{anytime valid test} which allows arbitrary testing policies that need not necessarily be specified in advance.

\begin{definition}[anytime validity]
	Let a null hypothesis $\Hnull$ be given. We say that a test is \emph{anytime valid}, or has \emph{Type-I error control}, at a level $\alpha\in(0,1)$ if for all $\Fcal$-stopping-times $\tau$ we have
	\begin{equation*}
		P( \text{$\Hnull$ is rejected at time $\tau$} ) \leq \alpha
		\text{ for all $P\in\Hnull$.}
	\end{equation*}
\end{definition}

%In the above definition, we interpreting the stopping time $\tau$ as representing the time or iteration after which we decide to reject the hypothesis $\Hnull$.
%We note that, by definition, anytime validity guarantees that a test has Type-I error bounded by $\alpha$ within the sequential testing setting. A related definition, concerned with the Type-II error of a test, is that of its power. Precisely, the power of a test is the probability that a test can correctly reject a null hypothesis when it is false. That is, the power of a test $\tau$ with respect to a measure $Q\neq\Hnull$ is simply $Q( \text{ $\tau$ rejects $\Hnull$} )$,
%the probability of rejecting $\Hnull$ given that the null is false. Ideally, we wish to design tests $\tau$ for which the power is as large a possible, which translates to a test which rejects false null hypotheses with high probability.

% \begin{definition}[Power]
% 	Let $\Hnull$ be a null hypothesis and $\tau$ an $\Fcal$-adapted test for this hypothesis. For any alternative hypothesis $Q \in \Mcal_1(\Xcal^\N)\setminus \Hnull$, we define the power of $\Tcal$ with respect to $Q$ as
% 	\begin{equation*}
% 		\sup_{t\in\N} Q( \text{ $\Tcal_t$ rejects $\Hnull$} )
% 		\;,
% 	\end{equation*}
% 	the probability that $\Tcal$ ever rejects $\Hnull$ under the measure $Q$.
% \end{definition}
% We typically seek tests with high power due to their ability to quickly reject null hypotheses whenever they are false.

A classical method for constructing anytime valid tests is based on nonnegative supermartingales. The following definition goes back to \cite{MR2849911}.

%Using nonnegative supermartingales, it turns out to be quite easy to construct anytime-valid tests for arbitrary null hypotheses. In particular these tests make use of Ville's inequality to provide them withe anytime validity propery. Before going though with the construction of these tests, however, we present the definition of a `test supermartingale' and Ville's inequality below.
%\rbox{mention e-values} \rbox{anytime valid tests are necessarily generated by supermartingales, cite ML}

\begin{definition}[test supermartingale]
	Let a null hypothesis $\Hnull$ be given. A \emph{test supermartingale (for $\Hnull$)} is a nonnegative adapted process $W := (W_t)_{t\in\N}$ with initial value $W_0 \le 1$ that is a $P$-supermartingale for all $P\in\Hnull$.
\end{definition}

We recall that a random process $(Z_t)_{t\in\N}$ adapted to a filtration $(\Gcal_t)_{t\in\N}$ such that $\E_P|Z_i| < \infty$ is a $P$-supermartingale ($P$-submartingale) if for all $t\in\N$ we have $\E_P[ Z_{t+1} \mid \Gcal_t ] \leq Z_t$ ($\E_P[ Z_{t+1} \mid \Gcal_t ] \geq Z_t$). A process that is both a $P$-supermartinale and a $P$-submartingale is called a $P$-martingale, and satisfies the equality $\E_P[ Z_{t+1} \mid \Gcal_t ] = Z_t$ for all $t \in \N$. Here, $\E_P$ means that the expectation is taken with respect to the probability measure $P \in \Mcal_1(\Xcal^\N)$.

Test supermartingales can be used to construct anytime valid tests. The basic tool for showing anytime validity is Ville's inequality \citep{MR3533075}, which states that any nonnegative $P$-supermartingale $W$ with $W_0 \le 1$ satisfies $P\left( \sup_{t\in\N} W_t > 1/\alpha \right) \leq \alpha$ for all $\alpha \in (0,1)$. Thus, if $W$ is a test supermartingale and $\alpha \in (0,1)$ is fixed, then the test which rejects the null as soon as $W$ reaches a value above $1/\alpha$,
\begin{equation} \label{eq:test-smg-test}
	\tau = \inf\left\{ t\in\N \colon W_t > \frac{1}{\alpha} \right\},
\end{equation}
satisfies
\begin{align*}
	P( \text{$\tau$ rejects $\Hnull$} ) 
	= P( \tau < \infty )
	= P\left( \sup_{t\in\N} W_t > \frac{1}{\alpha} \right) \leq \alpha,
\end{align*}
and is therefore anytime valid at level $\alpha$.

\begin{remark} \label{R_e_process}
A notion closely related to test supermartingales is that of an \emph{e-process}, which is a nonnegative adapted process $W$ such that $\E_P[W_\tau] \le 1$ for all $P \in \Hnull$ and all stopping times $\tau$ \citep{10.1214/20-AOS2020,xu_wang_ramdas_21,ram_ruf_lar_koo_20,MR4364897}. The stopping theorem implies that every test supermartingale is an e-process, but the converse is not true \citep{ram_ruf_lar_koo_20,MR4364897}. The `static' or non-sequential analog of an e-process is known as an \emph{e-variable}, which is a nonnegative random variable $E$ such that $\E_P[ E ] \le 1$ for all $P \in \Hnull$. These notions have recently been studied extensively as a tool for safe inference \citep{gru_hei_koo_19}.
\end{remark}

\subsection{Power and growth}

In addition to validity we are also interested in power against suitable alternative hypotheses $\Hcal_1$ disjoint from $\Hcal_0$. Loosely speaking this means that if the true data distribution belongs to $\Hcal_1$, the test should reject $\Hcal_0$ quickly with high probability. For tests arising from test supermartingales $W$ via \eqref{eq:test-smg-test} this corresponds to designing $W$ to grow quickly with high probability under distributions in $\Hcal_1$. This can be approached using the Growth Rate Optimal (GRO) criterion, which has recently received significant attention in the context of e-values and e-process \citep{gru_hei_koo_19}. In our setting the GRO criterion is as follows. At each time $T$ one seeks to maximize the expected logarithmic increment conditionally on data observed so far across all test supermartingale increments. More formally one aims to solve
\begin{equation}\label{eq_GROW}
\maximize \ \E_Q \left[ \log \frac{W_{T+1}}{W_T} \mid \Fcal_T \right] \quad \text{ subject to } \quad \sup_{P \in \Hnull} \E_P\left[ \frac{W_{T+1}}{W_T} \mid \Fcal_T \right] \le 1
\end{equation}
given the observed data $X_1,\ldots,X_T$, where $Q$ is a suitable distribution. As we will see, $Q$ need not itself be the only element of $\Hcal_1$, or even belong to $\Hcal_1$ at all. It is a purely computational device used to guide the choice of $W$.

In implementing this idea one is faced with three key issues:
\begin{enumerate}
\item\label{issue_1} The problem \eqref{eq_GROW} optimizes over the set of all test supermartingales. Solving it requires a description of this set, or of a sufficiently rich subset.
\item\label{issue_2} A suitable distribution $Q$ has to be specified.
\item\label{issue_3} One has to actually solve \eqref{eq_GROW}, at least numerically, and ideally derive performance guarantees with respect to the set $\Hcal_1$ of alternatives.
\end{enumerate}
In this paper, we consider null hypotheses $\Hnull$ based on elicitable functionals and identifiable functionals, which admit large families of explicit test supermartingales. This addresses \ref{issue_1}. In order to address \ref{issue_2} we focus on distributions $Q$  which are not fixed in advance but rather learned in an online fashion as more and more data is observed. This idea has recently also been explored by \cite{wau_ram_20}. Having dealt with \ref{issue_1} and \ref{issue_2}, the GRO criterion becomes a concrete optimization problem which we solve using methods from Online Convex Optimization (OCO). A key feature of this approach is that OCO methods come with asymptotic performance guarantees in the form of regret bounds. We employ these bounds to show that the resulting tests have asymptotic power one under a large composite nonparametric alternative hypothesis $\Hcal_1$, in the sense that we obtain a test supermartingale which tends to infinity with probability one under every distribution in $\Hcal_1$. Consequently, if the true data distribution is some element of $\Hcal_1$ then any test $\tau$ of the form \eqref{eq:test-smg-test} is guaranteed to eventually reject the null: $P( \tau < \infty) = 1$ for all $P \in \Hcal_1$. This addresses \ref{issue_3}.

\subsection{Test supermartingales via mixing}
\label{sec:betting-against-the-null}

The null hypotheses considered in this paper will be constructed directly in terms of explicit families of test supermartingales $L^\theta = (L^\theta_t)_{t \in \N}$ indexed by a parameter $\theta \in \Theta$, where $\Theta$ is an (arbitrary) index set. Whenever such a family $\{ L^\theta \}_{\theta \in \Theta}$ is available, it is possible to construct new test supermartingales by combining its members. For instance, it is clear that any convex combination of test supermartingales is again a test supermartingale. More generally, one can use predictably mixed test supermartingales as shown in the following lemma; see also \cite{wau_ram_20}. In this way one can assemble weak test supermartingales into more powerful ones. %An important feature of these `composite' processes, built on top of a base family of processes, is that they can combine weak tests to constuct more powerful ones.

\begin{lemma}[Predictably mixed supermartingale] \label{lem:predictably-controlled-test-smg}
Let $\{L^\theta\}_{\theta\in\Theta}$ be a family of test supermartingales and $(\pi_t)_{t\in\N}$ a predictable sequence of probability measures on $\Theta$.
%\footnote{This means that for each $t \in \N$, $\pi_t = \pi_t(X_1,\ldots,X_{t-t}; d\theta)$ is a probability measure on $\Theta$ that may depend on the preceding data points.} 
Then the process $W=(W_t)_{t\in\N}$ defined by $W_0 = 1$ and
	\begin{equation} \label{eq:predictably-controlled-test-smg-1}
		W_t = \prod_{i = 1}^t \int_{\Theta} \frac{L_i^\theta \;}{L_{i-1}^\theta} \, \pi_i(d\theta), \quad t \in \N,
	\end{equation}
	is also a test supermartingale.
\end{lemma}

To be precise, we assume here that $\Theta$ is a measurable space, and that $(\omega,\theta) \mapsto L_t^\theta(\omega)$ is measurable for each $t$. The condition on $(\pi_t)_{t\in\N}$ means that for each $t \in \N$, $\pi_t = \pi_t(X_1,\ldots,X_{t-t}; d\theta)$ is a probability measure on $\Theta$ that may depend on the preceding data points. The proof of Lemma \ref{lem:predictably-controlled-test-smg} can be found in the appendix.

\begin{example}[i.i.d.\ Gaussians]
Consider the hypothesis $\Hnull = \{ P \in \Mcal_1(\Xcal^\N)\colon X_t\sim \text{ i.i.d. }   \Ncal(0,1) \}$. Then for every $\lambda\in\R$ the process $L^\lambda$ defined by $L^\lambda_0 = 1$ and for $t \in \N$ by
%	\begin{equation*}
		$L_t^\lambda = \exp\left( \lambda \sum_{i = 1}^t X_i - \frac{1}{2} \lambda^2 t \right)$
%	\end{equation*}
is a test (super)martingale. The predictably mixed process $W$ in \eqref{eq:predictably-controlled-test-smg-1} then takes the form
\[
		W_t = \prod_{i = 1}^t \int_\R \exp\left( \lambda X_i - \frac{1}{2} \lambda^2 t \right) \pi_i(d \lambda).
\]
\end{example}

\begin{remark}\label{rem_trading_interpretation}
The mixing construction~\eqref{eq:predictably-controlled-test-smg-1} admits a useful interpretation in terms of trading a portfolio of financial assets. Treating the collection of test supermartingales $\{ L^\theta \}_{\theta\in\Theta}$ as a collection of tradable assets indexed by $\theta\in\Theta$, we may think of $W_t$ in~\eqref{eq:predictably-controlled-test-smg-1} as the value of a portfolio trading these assets. Indeed, regard $\pi_t(d\theta)$ as the portfolio weights specifying the proportion of capital allocated to each asset $\theta$ at time $t-1$; this is observable at time $t-1$ because $(\pi_t)_{t \in \N}$ is  a predictable sequence. The portfolio return from time $t-1$ to $t$ is the weighted average of the individual asset returns,
\[
\int_\Theta \frac{L_t^\theta - L_{t-1}^\theta}{L_{t-1}^\theta} \, \pi_t(d\theta).
\]
Rearranging \eqref{eq:predictably-controlled-test-smg-1} one sees that this is equal to the overall portfolio return $(W_t - W_{t-1}) / W_{t-1}$.
%Indeed, rearranging the terms in \eqref{eq:predictably-controlled-test-smg-1} we get
%\begin{equation} \label{eq:predictably-controlled-test-smg-2}
%	V_t = \prod_{i = 1}^t
%	\left(
%	1 + \int_{\Theta} \frac{L_i^\theta - L_{i-1}^\theta}{L_{i-1}^\theta} \, \pi_i(d\theta) 
%	\right),
%\end{equation}
%where $\nicefrac{(L_i^\theta - L_{i-1}^\theta)}{L_{i-1}^\theta}$ is the return of asset $\theta$ in period $i$. Each of the $\pi_i$ can be seen as portfolio weights specifying the proportion of capital allocated to each asset $\theta$ in period $i$. Hence, $V_t$ can be seen as the wealth process for a portfolio trading the assets $\{ L^\theta \}_{\theta\in\Theta}$, where the investor has an asset allocation strategy given by $(\pi_i)_{i\in\N}$. Each allocation strategy $\pi_t$ must be measurable with respect to $\Fcal_{t-1}$, which, using an abuse of notation, can be interpreted as requiring that $\pi_t(d\theta) = \pi_t(x_1,x_2,\dots,x_{t-1};d\theta)$ be a function of past data.
We can think of selecting a strong allocation strategy $(\pi_t)_{t\in\N}$ as choosing bets against $\Hnull$ in order to make $W_t$ grow quickly, eventually exceeding the threshold $\alpha^{-1}$ to reject $\Hnull$. Indeed, if $\Hnull$ is false, there may exist assets $L^\theta$ which are not supermartingales, enabling one to `bet against the null' by selecting $(\pi_t)_{t\in\N}$ with weights on these processes such that the wealth process grows on average. In contrast, if $\Hnull$ is true, Ville's inequality shows that, regardless of the trading strategy employed, it is unlikely (with probability bounded by $\alpha$) that our wealth ever exceeds the threshold $\alpha^{-1}$.
\end{remark}

In practice it is not feasible to work with general predictable sequences $(\pi_t)_{t \in \N}$. Instead we consider parsimonious specifications that tend to work well in experiments. A key example is the Dirac specification,
%\begin{equation*}
	$\pi_t = \delta_{\theta_t}$,
%\end{equation*}
where $(\theta_t)_{t\in\N}$ is a $\Theta$-valued predictable process. This is the simplest possible specification. In terms of the trading interpretation in Remark~\ref{rem_trading_interpretation}, a strategy of this kind chooses in each period one single asset where all capital is invested. The test supermartingale \eqref{eq:predictably-controlled-test-smg-1} simplifies to
\begin{equation} \label{eq:product-martingale-Dirac}
	W_t = \prod_{i=1}^t \frac{L_i^{\theta_i}}{L_{i-1}^{\theta_i}}.
\end{equation}

\section{Specifying the null hypothesis} \label{sec_null_hypothesis}
%\section{Elicitable and identifiable functionals}

We consider null hypotheses involving the value of certain statistical functionals of the (conditional) distributions of the data. For example, for a given value $\lambda_0$ we may want to to test the hypothesis
\[
\Hnull = \left\{P \in \Mcal_1(\Xcal^\N)\colon \text{$\lambda_0$ is a median of the conditional distribution $P(X_t \in \cdot \mid \Fcal_{t-1})$ for $t \in \N$} \right\}.
\]
The hypotheses considered below generalize this example beyond medians to a large class of \emph{elicitable functionals} and \emph{identifiable functionals}. These concepts are reviewed below; they include quantiles, moments, expectiles, and many other examples. The key common feature of these hypotheses is that they can be expressed in the form
\begin{equation} \label{eq_sup_mg_hypothesis}
\Hnull = \left\{ P \in \Mcal_1(\Xcal^\N)\colon L^\theta \text{ is a $P$-supermartingale for all } \theta \in \Theta \right\}
\end{equation}
for some explicit family of nonnegative processes $L^\theta = (L^\theta_t)_{t \in \N}$ starting at $L^\theta_0 = 1$, indexed by a parameter $\theta \in \Theta$ where $\Theta$ is an (arbitrary) index set. In our applications $\Theta$ will be a subset of a finite-dimensional space. Thus by construction, $\{L^\theta\}_{\theta \in \Theta}$ constitutes a family of `base' test supermartingales for $\Hcal_0$ which can be used to form other test  supermartingales through predictable mixing as explained in Subsection~\ref{sec:betting-against-the-null}.

\subsection{Definition of elicitability and identifiability}
\label{sec_def_elicitability_identifiability}

We review the definitions as given in \cite{FisslerZiegel2016}. 
%We briefly review the notions of \emph{elicitable functionals} and \emph{identifiable functionals}; see e.g.\ \rbox{references} for more information. 
Fix $k \in \N$ and a subset $\Lambda \subseteq \R^k$. A \emph{scoring function} is simply a measurable map $s \colon \Lambda \times \Xcal \to \R$. Let $\mathcal{M}$ be a class of probability distributions on $\mathcal{X}$. If for each distribution $\mu \in \mathcal{M}$ the map
\begin{equation}\label{eq:exp_score}
\lambda \mapsto \E_{\mu}[ s(\lambda, X) ]
%\lambda \mapsto \E_{x \sim \mu}[ s(\lambda, x) ]
\end{equation}
is well-defined and finite, let $T(\mu)$ denote the set of its minimizers. The induced map $T$ is called an \emph{elicitable functional (with respect to $\mathcal{M}$)} and $s(\lambda,x)$ a \emph{strictly consistent scoring function} for $T$. Here $X$ denotes the canonical random variable on $\Xcal$. If $T(\mu)$ consists only of one element, that is the minimizer in \eqref{eq:exp_score} is unique, we abuse notation and also use the notation $T(\mu)$ for the minimizer. Any given elicitable functional can have many different strictly consistent scoring functions.

Similarly, an \emph{identification function} is a measurable map $m \colon \Lambda \times \Xcal \to \R^k$. If for each distribution $\mu \in \mathcal{M}$ the map
\begin{equation}\label{eq:defidfunc}
\lambda \mapsto \E_{\mu}[ m(\lambda, X) ]
%\lambda \mapsto \E_{x \sim \mu}[ m(\lambda, x) ]
\end{equation}
is well-defined and finite, let $T(\mu)$ denote the set of its zeros. Then $T$ is called an \emph{identifiable functional (with respect to $\mathcal{M}$)} and $m(\lambda,x)$ a \emph{strict identification function} for $T$. 
%We note here that, in contrast with some definitions of identifiability, we explicitly assume that $m(\lambda, x)$ takes values in $\R^k$ in order to rule out the use of set-valued identification functions. 
A zero of the expected identification function in \eqref{eq:defidfunc} is understood to hold component-wise, since $m(\lambda,X)\in\R^k$. These definitions of elicitability and identifiability are called higher-order elicitability by \cite{FisslerZiegel2016}. 

When applying an elicitable or identifiable functional $T$ to a distribution $\mu$ in the following, we always implicitly assume that the functional is well-defined for this distribution. 

Table~\ref{tab:elicitable-identifiable-quantities} contains some examples of commonly used functionals that happen to be both elicitable and identifiable. The presented scoring functions are standard but not strictly consistent on the maximal possible domain $\mathcal{M}$ of definition of the respective functionals. Different choices of strictly consistent scoring functions allow to show elicitability of these functionals on their natural domains of definition; see \cite{gneiting2011making} for details.

% We note here that definitions provided above correspond to what is known as \emph{higher-order} elicitability and identifiability within the QRM literature\todoPC{cite sources}. 
% To get a sense of the link between scoring functions and identification functions, note that if $s(\lambda,x)$ is strictly convex in $\lambda$ and has a uniformly bounded gradient, then it is strictly consistent for a functional $T$ if and only if its gradient $m(\lambda,x) = \nabla_\lambda s(\lambda,x)$ is a strictly consistent identification function for $T$. We refer to \rbox{Gneiting, Fissler--Ziegel, etc.} for more details.

In the presence of convexity, the connection between scoring and identification functions can be understood through subgradients.
For instance, if $T$ is an elicitable functional whose scoring function $s(\lambda,x)$ is convex in $\lambda$, then $T$ is also an identifiable functional with identification function $m(\lambda, x) \in \partial_\lambda s(\lambda,x)$, an element of the subgradient of the scoring function with respect to $\lambda$. For the converse, linking an identification function $m(\lambda, x)$ to a unique convex scoring function $s(\lambda,x)$, more subtle conditions are needed, and we point interested readers to~\citet[Theorem 12.25]{rockafellar2009variational}. In the absence of convexity, scoring and identification functions are still linked through gradients under sufficient differentiability assumptions which are formalized as Osband's principle in \cite{FisslerZiegel2016}.

\begin{table}[ht]
% \centering
\begin{tabular}{@{}rccc@{}}
\toprule
\textbf{} & \textbf{$T(\mu)$} & \textbf{$s(\lambda,x)$} & \textbf{$m(\lambda,x)$} \\ \midrule
\textbf{Mean} & $\E_\mu [ X ]$ & $\frac{1}{2}(x-\lambda)^2$ & $x-\lambda$ \\
\textbf{$\alpha$-Quantile} & $\inf\{ \lambda:\alpha \geq \mu( X < \lambda ) \} $ & $|x-\lambda| \left( \, \alpha \mathbbm{1}_{\{x<\lambda\}} + (1-\alpha) \mathbbm{1}_{\{x>\lambda\}} \, \right)$ & $\mathbbm{1}_{\{x>\lambda\}} - \alpha$ \\
\textbf{Regression} & $\argmin_{\lambda\in\R^k} E_\mu[ (\lambda' Y - Z)^2 ]$ & $\frac{1}{2}\| \lambda'y  - z \|^2$ & $\lambda y y' - z y'$ \\ \bottomrule
\end{tabular}
\caption{Examples of statistical quantities that can be expressed as elicitable and identifiable functionals. In the elicitable case, for the mean, $\mathcal{M}$ is the class of all distributions with finite second moment; for the quantiles, it is the class of all distributions with finite first moment; for regression, $X = (Y,Z), x = (y,z) \in \R^{k+1}$, and $\mathcal{M}$ contains all distributions on $\R^{k+1}$ with finite expected squared norm. In the identifiable case, for the mean, $\mathcal{M}$ is the class of all distributions with finite first moment; for the quantiles, it is the class of all distributions with continuous distribution function at the $\alpha$-quantile; for regression, $X = (Y,Z), x = (y,z) \in \R^{k+1}$, and $\mathcal{M}$ contains all distributions on $\R^{k+1}$ with finite mean.}
\label{tab:elicitable-identifiable-quantities}
\end{table}

%Importantly, hypotheses based on elicitable functionals $T$ of the form of definition~\ref{def:elicitable-functional} are covered by \eqref{eq_EVs_1}. Indeed, if $s(\lambda,x)$ is a scoring function for $T$ and $\lambda_0$ is a particular value that $T$ may take, then the null hypothesis that $\lambda_0 \in T(\mu_0)$ can be encoded by taking $\Theta = \Lambda$ and $f_\theta(x) = s(\lambda_0,x) - s(\theta,x)$. Here $\mu_0$ is the true data generating distribution. Similarly, suppose $T$ is an identifiable functional of the form of Definition~\ref{def:identifiable-functional} with an identification function $m(\lambda,x)$ with components $m_1(\lambda,x),\ldots,m_k(\lambda,x)$. Then the null hypothesis that $\lambda_0 \in T(\mu_0)$ can be encoded by using \eqref{eq_EVs_1}, with an equality rather than an inequality, and taking $\Theta = \{1,\ldots,k\}$ and $f_\theta(x) = m_\theta(\lambda_0,x)$ for $\theta = 1,\ldots,k$. In this way, the testing problem in Section~\ref{sec:testing-smg-hypotheses} can be viewed as an abstraction of the testing problem for elicitable and identifiable functionals.

We now use the concepts of elicitability and identifiability to construct null hypotheses for the sequential testing problem. Let $T$ be either an elicitable functional or an identifiable functional with scoring function $s(\lambda,x)$ (in the elicitable case) or identification function $m(\lambda,x)$ (in the identifiable case). Given a fixed value $\lambda_0 \in \Lambda$ we consider the null hypothesis
\begin{equation} \label{eq_null_seq_T}
\Hnull = \left\{ P \in \Mcal_1(\Xcal^\N)\colon \lambda_0 \in T( P( X_t \in \cdot \mid \Fcal_{t-1})) \text{ for all } t \in \N, \text{ $P$-a.s.} \right\}.
\end{equation}
Thus the hypothesis is that $T$ returns a set containing $\lambda_0$ whenever it is applied to the conditional distribution of an observation $X_t$ given all earlier observations. For instance, if $T$ is the median functional, we recover the example at the beginning of this section.

Using the definition of elicitability or identifiability, we obtain a simpler representation of $\Hnull$ in terms of supermartingales or martingales and the scoring or identification function. Specifically, we have
\begin{equation} \label{eq_elicitable_null}
\Hnull = \left\{ P \in \Mcal_1(\Xcal^\N)\colon \sum_{i=1}^t \left( s(\lambda_0, X_i) - s(\lambda, X_i) \right) \text{ is a $P$-supermartingale for all } \lambda \in \Lambda \right\}
\end{equation}
in the elicitable case, and
\begin{equation} \label{eq_identifiable_null}
\Hnull = \left\{ P \in \Mcal_1(\Xcal^\N)\colon \sum_{i=1}^t m(\lambda_0, X_i) \text{ is a $P$-martingale} \right\}
\end{equation}
in the identifiable case. Note that in the identifiable case, $\sum_{i=1}^t m(\lambda_0,x_i)$ is a \emph{vector valued} martingale. For convenience, if $\Hnull$ is of the form \eqref{eq_elicitable_null} we call it an \emph{elicitable hypothesis}, and if it is of the form \eqref{eq_identifiable_null} we call it an \emph{identifiable hypothesis}.

%\begin{remark}
Let us spell out why, in the elicitable case, \eqref{eq_null_seq_T} essentially coincides with \eqref{eq_elicitable_null}; the identifiable case is similar. Due to the definition of elicitability, $\lambda_0 \in T( P( X_t \in \cdot \mid \Fcal_{t-1}))$ is equivalent to having $\E_P[s(\lambda_0,X_t) \mid \Fcal_{t-1}] \le \E_P[s(\lambda,X_t) \mid \Fcal_{t-1}]$ for all $\lambda \in \Lambda$. This holds for all $t \in \N$ if and only if the process $\sum_{i=1}^t ( s(\lambda_0, X_i) - s(\lambda, X_i) )$ is a $P$-supermartingale. Thus the right-hand sides of \eqref{eq_null_seq_T} and \eqref{eq_elicitable_null} are essentially the same. There is one subtlety that we have neglected in this argument which is usually irrelevant in applications. Since strictly consistent scoring functions are not unique, it may happen that the elicitable functional is defined for a larger class of distributions than the one where the chosen strictly consistent scoring function in \eqref{eq_elicitable_null} has finite expectation. This means that the moment conditions on the conditional distributions in \eqref{eq_elicitable_null} may be slightly stronger than in \eqref{eq_null_seq_T}. However, for many examples including the ones in Table \ref{tab:elicitable-identifiable-quantities}, this problem does not arise since we work with score differences.
%\end{remark}

%\begin{remark}
	The scoring function which elicits a functional $T$ is usually not unique. For example, the class of consistent scoring functions for the mean consists of all Bregman loss functions \citep{Savage1971}. %In particular,, amongst others, show that the score functions, which elicit functionals which are ratios of expected values, can be expressed as a parametric family of of Bregman divergences.
	Although there are no general guidelines for how one should select a scoring function, it is often natural to give preference to scoring functions that satisfy certain additional desirable properties, a relevant example in our setting being convexity in the first argument. For the mean, ratios of expectations and quantiles, convex strictly consistent scoring functions are essentially unique, see \citet[Corollary 4.2.17]{Fissler2017}, \cite{caponnetto2005note} and \citet[Corollary 10]{steinwart2014elicitation}. In the context of estimation in semi-parametric models for a quantile or the mean, \cite{komunjer2010semiparametric,komunjer2010efficient,dimitriadis2020efficiency} show that there exist unique choices of scoring functions which maximize the asymptotic efficiency of the estimators, but these are different from the convex choices described above.
%\end{remark}

The form of $\Hnull$ in \eqref{eq_elicitable_null} and \eqref{eq_identifiable_null} are suggestive of how one could construct families $\{L^\theta\}_{\theta \in \Theta}$ of base test supermartingales. If $s(\lambda,x)$ or $m(\lambda,x)$ is uniformly bounded, this is straightforward, see Subsection \ref{sec:uniformly-bounded-hypotheses}. The unbounded case requires to include additional moment bounds, which we achieve by imposing a sub-$\psi$ condition, see Subsection~\ref{sec:sub-psi-hypotheses}. %and~\ref{sec_unbounded_elicitable_ident_hyp}. 

% of nonnegative supermartingales which can be in turn used to construct hypothesis tests. Under additional assumptions on the underlying data as well as the the scoring and identification functions involved, we can produce such families of supermartingales. We concern ourselves with two distinct cases, the first involving hypotheses on uniformly bounded scoring and identification functions, and the second involving scoring and identification functions which have upper bounded cumulant generating function. 

\subsection{Uniformly bounded scoring and identification functions}
\label{sec:uniformly-bounded-hypotheses}

We construct parametric families $\{L^\theta\}_{\theta\in\Theta}$ of test martingales when the score difference, $(\lambda,x) \mapsto s(\lambda_0,x) - s(\lambda,x)$, or the norm of the identification function, $x \mapsto m(\lambda_0,x)$, are uniformly bounded. %In the lemmas that follow, we present test martingale constructions separately for the elicitable and identifiable cases. 
The proofs of the following two lemmas can be found in the appendix.

\begin{lemma}[Test martingales for elicitable hypotheses]
	\label{lem:test-smg-family-bded-elic}
	Consider an elicitable hypothesis $\Hnull$ of the form~\eqref{eq_elicitable_null}, and assume that $\,\inf_{\lambda\in\Lambda, x\in\Xcal} \{ s(\lambda_0,x) - s(\lambda,x) \} > -1$.
	For each $\lambda\in\Lambda$, define the processes $L^\lambda = (L^\lambda_t)_{t\in\N}$ by
	\begin{equation*}
		L^\lambda_t = \prod_{i=1}^t \left( 1 + s(\lambda_0,X_i) - s(\lambda,X_i) \right)
		\;.
	\end{equation*}
	Then the collection of processes $\{ L^\lambda \}_{\lambda\in\Lambda}$ forms a family of $\Hnull$ test supermartingales.
\end{lemma}

The lower bound of $-1$ appearing in the assumption that $\inf_{\lambda\in\Lambda, x\in\Xcal} \{ s(\lambda_0,x) - s(\lambda,x) \} > -1$ is without loss of generality since scoring functions may be rescaled by positive constants leaving the elicitable functional itself unchanged. 
In a similar fashion, we may construct test martingales for identifiable hypotheses as follows.

\begin{lemma}[Test martingales for identifiable hypotheses]
	\label{lem:test-smg-family-bded-ident}
 	Consider an identifiable hypothesis $\Hnull$ of the form~\eqref{eq_identifiable_null}, assume that $\sup_{x\in\Xcal}\| m(\lambda_0,x) \| < \infty$ and define $\Abd_{m,\lambda_0} := \{ \eta \in \R^d : \inf_{x\in\Xcal} \langle \eta , m(\lambda_0,x) \rangle > -1 \}$. For each $\eta\in \Abd_{m,\lambda_0}$, define the process $L^\eta = (L^\eta_t)_{t\in\N}$ by
	\begin{equation*}
		L^\eta_t = \prod_{i=1}^t \left( 1 + \langle \eta \mathrel{,} m(\lambda_0,X_i)\rangle \right)
		\;.
	\end{equation*}
	Then $\Abd_{m,\lambda_0}$ is convex with non-empty interior, and the set $\left\{L^\eta\right\}_{\eta \in \Abd_{m,\lambda_0}}$ forms a family of $\Hnull$ test martingales.
\end{lemma}

% Again, we note that the right hand side of the upper bound, $\| m(\lambda_0,x) \| < 1$ can be replaced by any positive constant without loss of generality, since we may simply re-scale $m$ without changing the underlying identifiable functional. It is also worth pointing out that the boundedness condition in the identifiable case only involves $\lambda = \lambda_0$, whereas in the elicitable case it involves all $\lambda$.

\begin{remark} \label{rem:ident-vs-convex-bounded}
Whenever the scoring function $s(\lambda,x)$ is convex in $\lambda$ and satisfies the conditions of Lemma~\ref{lem:test-smg-family-bded-elic}, there is a direct connection between the two martingale constructions presented above. Indeed, if we let $m(\lambda_0,x) \in \partial_\lambda s(\lambda_0,x)$, then
\begin{equation} \label{eq:convex-elic-ident-relation}
	-1 < s(\lambda_0,x) - s(\lambda,x) \leq \langle \lambda_0 - \lambda , m(\lambda_0,x) \rangle
\end{equation}
for all $x$ and $\lambda$. Hence, each increment of the test martingale construction of Lemma~\ref{lem:test-smg-family-bded-ident} can be thought of as a linearization of the increments of the processes defined in Lemma~\ref{lem:test-smg-family-bded-elic}. Moreover, equation~\eqref{eq:convex-elic-ident-relation} implies that for any elicitable hypotheses with convex scoring function, identifiable test martingales generated by Lemma~\ref{lem:test-smg-family-bded-ident} will dominate the elicitable test martingale Lemma~\ref{lem:test-smg-family-bded-elic} whenever $\eta = \lambda-\lambda_0$. Indeed, it is easy to verify that 
\begin{equation*}
    0 < \prod_{i=1}^t ( 1 + s(\lambda_0,x_i) - s(\lambda,x_i) )
    \leq
    \prod_{i=1}^t ( 1 + \langle \lambda_0 - \lambda , m(\lambda_0,x_i) \rangle )
    \;,
\end{equation*}
and that the right-hand side produces a valid test supermartingale for all $\lambda,\lambda_0\in\Lambda$. This observation suggests that whenever an elicitable functional $T$ admits a bounded and convex scoring function, the test generated by its subgradient using Lemma~\ref{lem:test-smg-family-bded-ident} will always be more powerful that the one generated by Lemma~\ref{lem:test-smg-family-bded-elic}.
\end{remark}

\citet[Proposition 3.2.1]{Fissler2017} shows that, under suitable conditions, identification functions are unique up to multiplication with a matrix valued function in $\lambda$. Therefore, the Remark \ref{rem:ident-vs-convex-bounded} does not only apply to a subgradient of a convex scoring function but to any identification function, as long as a convex scoring function for the respective functional exists, and with a suitable modification of the relation $\eta = \lambda_0 - \lambda$.

In the setting of analyzing the asymptotic efficiency of semi-parametric estimators of elicitable and identifiable functionals, a similar relation is observed, in which an estimator generated by the identification function will always be asymptotically more efficient that its elicitable counterpart \citep{dimitriadis2020efficiency}.

The uniform boundedness assumptions in Lemmas~\ref{lem:test-smg-family-bded-elic} and~\ref{lem:test-smg-family-bded-ident} may appear to be restrictive. However, they cover a number of cases of interest including the mean whenever the data generating process $\{X_i\}_{i\in\N}$ is bounded, see also \cite{wau_ram_20}.
%it will yield both a bounded functional $T(\mu)$ and domain $\Lambda$, as well as uniformly bounded scoring or identification functions.
A second relevant example are (vectors of) quantiles, where the uniform boundedness assumption for the identification function is met, regardless of whether the data generating process is bounded or not. Indeed, it is easy to see from the rightmost column of Table~\ref{tab:elicitable-identifiable-quantities} that $\| m(\lambda,x) \| \leq \max\{ \alpha\mathrel{,}1-\alpha\}$ is uniformly bounded. Hence the family of test martingales of Lemma~\ref{lem:test-smg-family-bded-ident} is always valid in the case of testing quantiles.

%\subsection{Sub-\texorpdfstring{$\psi$}{psi} variables and processes}
\subsection{Test supermartingales for sub-\texorpdfstring{$\psi$}{ψ} hypotheses}
\label{sec:sub-psi-hypotheses}

In the more general case that the scoring function or identification function is unbounded, we construct families of test martingales under the assumption of a tail bound on the scoring or identification function which involves bounding the cumulant generating function. We introduce the definition of a sub-$\psi$ process below, a notion related to those introduced in~\cite{10.1214/aop/1176996452,10.1214/009117904000000397} but most closely related to~\cite[Definition~1]{howard2018timeuniform}.

\begin{definition}[Sub-$\psi$ Process] \label{def:sub-psi-rv}
    Let $Y=(Y_t)_{t\in\N}$ and $V=(V_t)_{t\in\N}$ be $\Fcal$-adapted processes, where the variance process $V_t$ is assumed to be non-negative and $Y_0=V_0=0$. We say that $(Y,V)$ is sub-$\psi$ if there is a $\umax>0$ and a nonnegative convex function $\psi:[0,\umax)\to[0,\infty)$ satisfying satisfying $\psi(0) = \psi^\prime(0) = 0$, where $\psi^\prime(0)$ is its right derivative, and for each $u\in[0,\umax)$,
    \begin{equation*}
        (e^{u \, ( Y_t - \bar Y_t ) - V_t \, \psi(u)})_{t \in \N}
    \end{equation*}
    is a supermartingale, where $\bar Y_t = \sum_{i=1}^{t} \E[Y_i - Y_{i-1} \mid \Fcal_{i-1}]$.
\end{definition}

% \begin{definition} \label{def:sub-psi-rv}
% 	Let $\umax\in (0,\infty]$ and $\psi : [0,{u}_\text{max}) \to [0,\infty)$ be a nonnegative convex function satisfying $\psi(0) = \psi^\prime(0) = 0$, where $\psi^\prime(0)$ is its right derivative. % and where we use the convention that $\psi(u) = +\infty$ outside $[0,u_\text{max})$. 
% 	We say that an integrable random variable $Y$ is \emph{sub-$\psi$} if
% 	\begin{equation*}
% 		\E\left[ e^{ u \,(Y - \E[Y]) } \right]  \le e^{\psi(u)} \text{ for all } {u} \in [0, {u}_\text{max}).
% 	\end{equation*}
% \end{definition}

%  A useful way of understanding the sub-$\psi$ condition is as the assumption that the data satisfies a right tail bound.
%\begin{remark}
 Definition~\ref{def:sub-psi-rv} is similar to \citet[Definition~1]{howard2018timeuniform}, but there are some noteworthy differences. In particular, \cite[Definition~1]{howard2018timeuniform} is a weaker condition in that it allows $e^{u \, ( Y_t - \bar Y_t ) - V_t \, \psi(u)}$ to only be upper-bounded by a supermartingale, rather than be a supermartingale itself. We make the choice of requiring the supermartingale condition in order to be able work with the supermartingale predictable mixing introduced in Section~\ref{sec:betting-against-the-null}, which would break down without this assumption.
 Further discussion of the sub-$\psi$ condition and its applications in time-uniform confidence bounds can be found in \cite{howard2018timeuniform,howard2020timeuniform}. In particular, we point the reader to~\citet[Section 3.1, Table 3]{howard2018timeuniform} for a collection of commonly used $\psi$ functions and variance processes $V_t$ which are valid under a wide variety of assumptions. 
%\end{remark}

Typically, $V_t=t$ is the simplest possible choice of variance process, and we use it for all concrete examples in this paper. We have chosen to state our theoretical results for the more general Definition \ref{def:sub-psi-rv} in order to be consistent with existing literature. When $V_t=t$, the sub-$\psi$ condition specializes to (conditional, one-sided versions of) sub-Gaussian, sub-Gamma, sub-Exponential, sub-Bernoulli and related conditions on the increments $\Delta Y_t = Y_{t+1} - Y_t$ of $Y$, obtained by choosing $\psi$ to be the corresponding cumulant generating function. Specifically, the condition in Definition  \ref{def:sub-psi-rv} is then equivalent to
\[
	\E\left[ e^{ u \,(\Delta Y_t - \E[\Delta Y_t\mid \Fcal_t]) }\mid \Fcal_{t-1}] \right]  \le e^{\psi(u)} \text{ for all $u \in [0, {u}_\text{max})$, $t \in \N$.}
\]
This condition implies a bound on the right tail probabilities of the increments of a sub-$\psi$ process $Y$. Indeed, Chernoff's inequality \cite[see e.g.][]{HAGERUP1990305} states that
\begin{equation*} \label{eq:cdf-bound-psi}
	% \log P\left(  Y > c + \E[Y] \right) \leq -\psi^{\ast}\left( c \right)
	\log P\left(  
	\Delta Y_t > c + \E[ \Delta Y_t \mid \Fcal_t ] 
	\middle|
	\Fcal_t
	\right) \leq -\psi^{\ast}\left( c \right)
	\text{ for all $t\in\N$},
\end{equation*}
where  $\psi^\ast(c) = \sup\{ u \, c - \psi(u) \colon u \in [0, u_\text{max}) \}$ is the convex conjugate of $\psi$.

The following lemma shows that, for a given sub-$\psi$ process $Y=(Y_t)_{t \in \N}$, the supermartingale property on $Y$ is equivalent to the existence of a non-negative supermartingale.

\begin{lemma} \label{lem:sub-psi-iff-neg-mean}
    Suppose that $(Y,V) = (Y_t,V_t)_{t\in\N}$ is an $\Fcal$-adapted sub-$\psi$ process. Then $Y$ is a supermartingale if and only if $(e^{u \, Y_t - V_t \psi(u)})_{t \in \N}$
    is a supermartingale.
\end{lemma}
\begin{proof}
   Let $Y$ be a supermartingale. By the sub-$\psi$ property, %we have that
 %   \begin{equation*}
 %       \E\left[ e^{u \, ( Y_t - \bar Y_t ) - V_t \, %\psi(u)} \middle| \Fcal_{t-1} \right]
 %       \leq 
 %       e^{u \, ( Y_{t-1} - \bar Y_{t-1} ) - V_{t-1} %\, \psi(u)}
 %       \;.
 %   \end{equation*}
 %   Multiplying both sides by $e^{u\bar{Y}_t}$ 
 and denoting the forward increments of processes as $\Delta Z_{t} = Z_{t+1} - Z_{t}$,
    %$e^{u \, \sum_{i=1}^{t} \E[Y_i - Y_{i-1} \mid \Fcal_{i-1}] }$ 
    we have that
    \begin{align*}
        e^{u\bar{Y}_t}\E\left[ e^{u \,  (Y_t-\bar{Y}_t) - V_t \psi(u)} \middle| \Fcal_{t-1} \right]
        &\leq 
        e^{u \, Y_{t-1} - V_{t-1} \psi(u)}
        \;
        e^{u \, \Delta \bar Y_{t-1} }
        \\ &=
        e^{u \, Y_{t-1} - V_{t-1} \psi(u)}
        \;
        e^{u \, \E[ \Delta Y_{t-1} \mid \Fcal_{t-1} ] }
        \leq
        e^{u \, Y_{t-1} - V_{t-1} \psi(u)}
        \;,
    \end{align*}
    where the last inequality follows since the supermartingale property implies that $\E[ Y_{t} - Y_{t-1} \mid \Fcal_{t-1} ] \leq 0$.
    
    Now, assume that $(e^{u \, Y_t - V_t \, \psi(u)})_{t \in \N}$ is a supermartingale for all $u\in[0,\umax)$. Since $\psi^\prime(0)=0$ %and letting $\Delta Y_{t-1} = Y_t - Y_{t-1}$ and $\Delta V_{t-1} = V_t - V_{t-1}$, 
    we have
    \begin{align*}
        \E[ \Delta Y_{t-1} \mid &\Fcal_{t-1} ]
        = \E\left[ \Delta Y_{t-1} - \Delta V_{t-1} \psi^\prime(0) \mid \Fcal_{t-1} \right]
        \\ &=
        \E\left[ 
        \lim_{u\downarrow 0} \frac{e^{u \, \Delta Y_{t-1} - \Delta V_{t-1} \, \psi(u)} - 1}{u}
        \middle| \Fcal_{t-1} \right]
        =
        \lim_{u\downarrow 0}
        \frac{ \E\left[e^{u \, \Delta Y_{t-1} - \Delta V_{t-1} \, \psi(u)}\mid \Fcal_{t-1} \right] - 1}{u}
        \;,
    \end{align*}
    where in the third line we use the fact that since $\psi$ is convex and differentiable at zero, there exists a closed neighborhood including zero in which $(e^{u \, \Delta Y_{t-1} - \Delta V_{t-1} \, \psi(u)} - 1)/u$ is continuous and hence has an integrable upper bound, allowing us to exchange the limit and the expectation by the dominated convergence theorem. 
    Lastly, noting that $( \E[e^{u \, \Delta Y_{t-1} - \Delta V_{t-1} \, \psi(u)}\mid \Fcal_{t-1} ] - 1 )/u \leq 0$ by the assumed supermartingale property, we conclude that $\E\left[ \Delta Y_{t-1} \mid \Fcal_{t-1} \right] \leq  0$, demonstrating that $(Y_t)_{t\in\N}$ is a supermartingale, as desired.
\end{proof}

%Hence, Lemma~\ref{lem:sub-psi-iff-neg-mean} shows that if $Y$ sub-$\psi$ process, then $Z=(Z_t)_{t\in\N}$ defined as $Z_t=\prod_{i=1}^t e^{Y_i - \psi(u)}$ is a super-martingale if and only if $Y$ is a supermartingale. This very observation will be the basis of the test supermartingale construction method presented in the next section.

%\fbox{Find home for this:} In Appendix~\rbox{fill this in}, we summarise some relevant examples of common settings where the score function or identification function norm are sub-$\psi$.

% \begin{remark}[Generalized Definition of a Sub-$\psi$ Process]
% \todoPC{@ML,JZ: This is an example of how the math changes a bit when adding in these new definitions. Most other theorems should just follow suit. I think changes to the remaining theorems should be relatively minimal though we would need to comb the text and modify the notation everywhere. It might also be useful to mention the example of a square integrable identification function, where the test becomes like a sort of t-test. }

% Following \citet[Table~3 \& Lemma~3]{howard2018timeuniform} we find that this definition of a sub-$\psi$ process allows us to generalize the concept of test supermartingales to elicitable or identifiable functionals with finite second or third moment and without necessarily finite exponential moments. Letting $V_t = t$, we arrive at the previous version of the definition.

% \end{remark}

We say that a family  $\{Y^\theta\}_{\theta\in\Theta}$ of integrable processes indexed by $\Theta$ is sub-$\psi$ if for each $\theta\in\Theta$ there is a function $\psi_\theta:[0,u_\text{max}) \to [0,\infty)$ and a process $V^\theta$ such that for each $\theta\in\Theta$, $Y^\theta$ is sub-$\psi_\theta$. We note that although $\theta\mapsto(\psi_\theta(\cdot),V^\theta)$ varies with $\theta$, the interval $[0,u_\text{max})$ is assumed to be the same for all $\theta\in\Theta$.
%\subsection{Test supermartingales for sub-\texorpdfstring{$\psi$}{ψ} hypotheses}
%\label{sec_unbounded_elicitable_ident_hyp}
Using Lemma~\ref{lem:sub-psi-iff-neg-mean}, we construct families of test supermartingales under the assumption that the scoring or identification functions satisfy a sub-$\psi$ condition, allowing us to extend the anytime-valid sequential testing methodology to unbounded data. The proofs of the following two Lemmas can be found in the appendix.
%In the lemmas that follow, we specify the precise assumptions and resulting test supermartingale families.

\begin{lemma}[Test supermartingales for sub-\texorpdfstring{$\psi$}{ψ} elicitable hypotheses]
	\label{lem:test-smg-family-subpsi-elic}
Let $T$ be an elicitable functional with scoring function $s(\lambda,x)$ and let $\Hcal_0$ be an elicitable hypothesis of the form~\eqref{eq_elicitable_null}. For every $\lambda\in\Lambda$ define $Y_t^\lambda = \sum_{i=1}^t s(\lambda_0,X_i) - s(\lambda,X_i)$ and $V_t^\lambda = \sum_{i=1}^t v^\lambda_i$ for some nonnegative $\Fcal$-adapted process $(v_i^\lambda)_{i\in\N}$. If the family $\{ (Y^\lambda,V^\lambda) \}_{\lambda\in\Lambda}$ is sub-$\psi$ under every measure in $\Hcal_0$, then for each $\lambda \in\Lambda$ and $u\in[0,u_\text{max})$, the process $L^{\lambda,u} = (L^{\lambda,u}_t)_{t\in\N}$ defined by
	\begin{equation*}
		L^{\lambda,u}_t
		= e^{u Y_t^\lambda - V_t^\lambda \psi_{\lambda}(u)}
		= \prod_{i=1}^t 
		e^{
		u \left( s(\lambda_0,X_i) - s(\lambda,X_i) \right)
		- v_i^\lambda \psi_\lambda(u)
		}
	\end{equation*}
is an $\Hcal_0$ test supermartingale.
\end{lemma} 

Proceeding almost identically to Lemma~\ref{lem:test-smg-family-subpsi-elic}, we may construct test supermartingale families for identifiable hypotheses.

\begin{lemma}[Test martingales for sub-\texorpdfstring{$\psi$}{ψ} identifiable hypotheses]
	\label{lem:test-smg-family-subpsi-ident}
Let $T$ be an identifiable functional with identification function $(\lambda,x) \mapsto m(\lambda,x) \in\R^k$, and let $\Hcal_0$ be an identifiable hypothesis of the form~\eqref{eq_identifiable_null}. Let $\Apsi_{m,\lambda_0} \subseteq \R^k$, and define for each $\eta\in\Apsi_{m,\lambda_0}$ the processes $Y_t^\eta = \sum_{i=1}^t \langle \eta \mathrel{,} m(\lambda_0,X_i)\rangle $ and $V_t^\eta = \sum_{i=1}^t v_i^\eta$ for some nonnegative $\Fcal$-adapted process $(v_i^\lambda)_{i\in\N}$. If the family of processes $\left\{\, (Y^\eta,V^\eta) \, \right\}_{ \eta \in \Apsi_{m,\lambda_0}}$ is sub-$\psi$ under every measure in $\Hcal_0$, then for each $\eta\in \Apsi_{m,\lambda_0}$ and $u\in[0,\umax)$ the process $L^{\eta,u} = (L^{\eta,u}_t)_{t\in\N}$ defined by
	\begin{equation*}
		L^{\eta,u}_t = \prod_{i=1}^t e^{ u
		\langle \eta , m(\lambda_0,X_i) \rangle - v_i^\eta \psi_{\eta}(u)
	    }
	\end{equation*}
	is an $\Hcal_0$ test supermartingale.
\end{lemma}

The following examples illustrate two situations where test supermartingales can be constructed for unbounded data using sub-$\psi$ assumptions.

%In order to provide the reader with some intuition on the sort of sub-$\psi$ assumptions required for testing elicitable and identifiable hypothesis, we provide some examples pertaining below relating to the statistics displayed in Table~\ref{tab:elicitable-identifiable-quantities}.

% \begin{example}[Sub-$\psi$ mean] \label{ex:scoring-function-mean}
%  	Recall from Table~\ref{tab:elicitable-identifiable-quantities} that a possible scoring function for the mean is $s(\lambda,x)=\frac{1}{2}(x-\lambda)^2$. We fix a value $\lambda_0$ and compute the difference $\Delta(\lambda,x)=s(\lambda_0,x) - s(\lambda,x) = \frac{1}{2}(\lambda_0^2 - \lambda^2) - (\lambda_0-\lambda) \, x$. 
%  	Assuming that $Y_t = \sum_{i=1}^t X_i$ is sub-$\psi$ with variance process $V_t=t$, we have that the family $\{ \sum_{i=1}^t \Delta(\lambda,X_i) \}_{\lambda \in \R}$ is sub-$\tilde\psi$ with $\tilde\psi_{\lambda}(\tilde u) = \psi((\lambda-\lambda_0)\tilde u) - \frac{\tilde u}{2} (\lambda^2 - \lambda_0^2)$ and $V_t=t$. If instead we consider the identification function for the mean, $m(\lambda,x) = x-\lambda$, we find the assumption that $\sum_{i=1}^t X_i$ is sub-$\psi$ with $V_t=t$ implies that condition on $\pm \sum_{i=1}^t m(\lambda_0,X_i)$ is sub-$\tilde{\psi}$ with $\tilde\psi_{\pm}(\tilde u)=\psi(\tilde u) - \tilde u \lambda_0$ and $V_t=t$.
%  \end{example} 

\begin{example}[Sub-$\psi$ mean] \label{ex:scoring-function-mean}
    Recall from Table~\ref{tab:elicitable-identifiable-quantities} that the mean is identifiable with identification function $m(\lambda,x) = x-\lambda$. Now suppose that a real-valued data-generating process $(X_t)_{t\in\N}$ has conditionally sub-Gaussian increments so that $\log \E[ e^{u (X_t - \E[X_t\mid \Fcal_{t-1}] )} \mid \Fcal_{t-1}] \leq \tilde \psi(u)$ for $\tilde \psi(u) = \sigma^2 u^2/2$ and some $\sigma>0$. Under this assumption, we have that $\sum_{i=1}^t \langle z , X_i - \lambda_0 \rangle$ is sub-$\psi$ with $\psi_z (u) = z^2 \tilde{\psi}(u)$ and $V_t = t$, allowing us to apply Lemma~\ref{lem:test-smg-family-subpsi-ident}.
    % On the other hand, if we just assume that $(X_t)_{t\in\N}$ has conditional variance $\bar v_t = \V[X_t\mid\Fcal_{t-1}]$, then
\end{example}

\begin{example}[Sub-$\psi$ regression] \label{ex:sub-psi-ark}
Consider the hypothesis that the data follows an $\mathrm{AR}(k)$ linear time series model, $X_t = \sum_{i=1}^k \beta_i \, X_{t-i} + \epsilon_t$, where $X_t, \epsilon_t \in \R$, $(\epsilon_t)_{t \in \N}$ is a martingale difference sequence where $(\pm \sum_{i=1}^t \epsilon_i)_{t\in\N}$ is sub-$\psi$ with variance process $V_t=t$ and $\beta \in \R^k$ is unknown. We wish to test whether $\beta = \beta_0 \in \R^k$. In each time step, $\beta$ is the value of the identifiable functional $T(P) = \argmin_{\beta} \E_P[ \| \sum_{i=1}^k \beta_i \, X_{t-i} - X_t \|^2 \mid \Fcal_{t-1} ]$. 
% Strictly speaking, the functional is defined on all distributions  the form $\delta_{x_1}\otimes \dots \otimes \delta_{x_k} \otimes Q$ for some probability measure $Q$ on $\R$ with finite second moment. 
In view of the regression example of Table~\ref{tab:elicitable-identifiable-quantities}, this functional has identification function
\begin{equation*}
    m\left(\beta \mathrel{,} (y,\mathbf{x}) \right) = 
	\left( \; 
		\left( \textstyle{\sum_{j=1}^k} \beta_j x_j- y \right) \frac{x_i}{\| \mathbf{x} \|}
	\; \right)_{i=1}^k
\end{equation*}
where $\mathbf{x}=(x_i)_{i=1}^k$ and where the re-scaling by $\nicefrac{1}{\|x\|}$ is possible because $\mathbf{X}_{(t - 1):(t-k)} = (X_{t-i})_{i=1}^k$ is $\Fcal_{t-1}$-measurable. 
% Since the $\epsilon_t$ may be unbounded, we apply our testing methodology in the sub-$\psi$ setting. We will show that a sufficient condition for this identifiable functional to satisfy the necessary sub-$\psi$ condition is that the signed residual processes $ (\pm \sum_{i=1}^t \epsilon_i )_{t\in\N}$ are sub-$\psi$ with variance process $V_t=t$.
In order to apply the testing methodology of Lemma~\ref{lem:test-smg-family-subpsi-ident}, we show that the processes $Y_t^z$ with $Y_{t}^z-Y_{t-1}^z = \langle z, m(\beta_0\mathrel{,}( X_t\mathrel{,}\mathbf{X}_{(t - 1):(t-k)} ) \rangle$
% \begin{align*}
% 	Y_{t}^z-Y_{t-1}^z 
% 	&= \left\langle z, m\left(\beta_0\mathrel{,}( X_t\mathrel{,}\mathbf{X}_{(t - 1):(t-k)} \right) \right\rangle
% 	\\ &= 
% 	(\langle\beta_0 \mathrel{,} \mathbf{X}_{(t - 1):(t-k)}\rangle - X_t ) \frac{\langle z, \mathbf{X}_{(t - 1):(t-k)}\rangle}{\| \mathbf{X}_{(t - 1):(t-k)} \|}
% % 	\\ &=
% % 	(\epsilon_t \, \beta_0 \cdot \mathbf{X}_{(t - 1):(t-k)} ) \frac{z \cdot \mathbf{X}_{(t - 1):(t-k)}}{\| \mathbf{X}_{(t - 1):(t-k)} \|}
% % 	\\ &=
% % 	\left( (\beta_0 - \beta) \mathbf{X}_{(t - 1):(t-k)} - \epsilon_t \right) \, \frac{\langle z, \mathbf{X}_{(t - 1):(t-k)}\rangle}{\| \mathbf{X}_{(t - 1):(t-k)} \|}
% 	=
% 	- \epsilon_t \, \frac{\langle z, \mathbf{X}_{(t - 1):(t-k)}\rangle}{\| \mathbf{X}_{(t - 1):(t-k)} \|}
% \end{align*}
are sub-$\psi_z$ with $V_t^z=t$ for all $z\in\R^d$ with $\|z\|\leq 1$. Computing 
$$\left\langle z \mathrel{,} m\left(\beta_0\mathrel{,}( X_t\mathrel{,}\mathbf{X}_{(t - 1):(t-k)} \right) \right\rangle - \E\left[ \left\langle z \mathrel{,} m\left(\beta_0\mathrel{,}( X_t\mathrel{,}\mathbf{X}_{(t - 1):(t-k)} \right) \right\rangle \middle| \Fcal_{t-1}\right] = c(z) \, \epsilon_t$$
where $c(z) = \langle z, \mathbf{X}_{(t - 1):(t-k)}\rangle/\| \mathbf{X}_{(t - 1):(t-k)} \|$, we see that if $\pm \sum_{i=1}^t \epsilon_i$ are sub-$\psi$ with $V_t=t$, then $Y_{t}^z$ will also be sub-$\psi$ since $\log\E_P\left[\exp( c(z) u \, \epsilon_t )\right] \leq \psi( |c(z)| \, u ) \leq \psi(u) $, which follows due to the fact that $|c(z)|\leq 1$. Hence, a sufficient condition for this identifiable functional to satisfy the necessary sub-$\psi$ condition is simply that the signed residual processes $(\pm \sum_{i=1}^t \epsilon_i)_{t\in\N}$ are sub-$\psi$.

\end{example}

%\begin{remark} \label{rem:ident-vs-convex-subpsi}
There is a relationship between tests for identifiable and elicitable hypotheses with a convex scoring function in the sub-$\psi$ case, analogous to the one pointed out in Remark~\ref{rem:ident-vs-convex-bounded}. Indeed, let $T(\mu)$ be an elicitable functional with a convex scoring function $(\lambda,x) \mapsto s(\lambda,x)$. As pointed out in Remark~\ref{rem:ident-vs-convex-bounded}, $T(\mu)$ is also identifiable with identification function $m(\lambda,x) \in \partial_{\lambda} s(\lambda,x)$. If we assume in addition that for a fixed $\lambda_0\in\Lambda$, the family of processes $\{ (  Y_t^\lambda )_{t\geq 0}  \}_{\lambda\in\Lambda}$ with increments $Y_t^\lambda - Y_{t-1}^\lambda = \langle \lambda_0 - \lambda \mathrel{,} m(\lambda_0,X_t) \rangle$ is sub-$\psi$ with $V_t=t$, then for any $\lambda \in \Lambda$ and $u\in[0,u_{\max})$, the process
\begin{equation*}
    L_t^{\lambda,u} = \prod_{i=1}^t e^{u \, \left\langle \lambda_0 - \lambda \,\mathrel{,}\, m(\lambda_0,X_i)) \right\rangle - \psi_\lambda(u)}
    \;,
\end{equation*}
is a valid test supermartingale according to Lemma~\ref{lem:test-smg-family-subpsi-ident}. However, since $s(\lambda,x)$ is convex and the $\left\langle \lambda_0 - \lambda \mathrel{,} m(\lambda_0,X_i) \right\rangle$ are increments of the sub-$\psi$ process $Y_t^\lambda$, we have that under $\Hcal_0$,
\begin{equation*}
    \E\left[ e^{ u \, \left( s(\lambda_0,X_i) - s(\lambda,X_i) \right) } \middle| \Fcal_{i-1} \right]
    \leq  \label{eq:sub-psi-convexity-inequality}
    \E\left[ e^{ u \, \left\langle \lambda_0 - \lambda \mathrel{,} m(\lambda_0,X_i) \right\rangle } \middle| \Fcal_{i-1} \right]
    \leq
    e^{\psi(u)} \notag
    \;.
\end{equation*}
Hence, the processes
\begin{equation*}
    \tilde{L}_t^{\lambda,u} =
    \prod_{i=1}^t e^{u \, \left( s(\lambda_0,X_i) - s(\lambda,X_i) \right) - \psi_\lambda(u)}
    \;,
\end{equation*}
are valid $\Hcal_0$ test supermartingales which match the form of the test supermartingales presented in Lemma~\ref{lem:test-smg-family-subpsi-elic}.
Due to~\eqref{eq:sub-psi-convexity-inequality}, however, we note that
 $\tilde{L}_t^{\lambda,u} \leq L_t^{\lambda,u}$. Hence, whenever $\lambda \mapsto s(\lambda,x)$ is convex and the families of processes $\left\{ \left( s(\lambda_0,X_i) - s(\lambda,X_i) \right)_{i\in\N} \right\}_{\lambda\in\Lambda}$ and $\left\{ \left( \langle\lambda_0 - \lambda \mathrel{,} m(\lambda_0,X_i) \rangle \right)_{i\in\N} \right\}_{\lambda\in\Lambda}$ are both increments of a sub-$\psi$ process, we find that the tests generated by the identification function $m$ according to Lemma~\ref{lem:test-smg-family-subpsi-ident} will always be more powerful than a test generated by the scoring function $s$ according to Lemma~\ref{lem:test-smg-family-subpsi-elic}, yielding a conclusion analogous to that in Remark~\ref{rem:ident-vs-convex-bounded}.
%\end{remark}

\begin{remark}[Bridging the sub-$\psi$ and bounded test supermartingales] \label{rem:connection-bounded-and-subpsi}
Although the sub-$\psi$ and uniformly bounded hypothesis testing methodologies may appear disjoint, there are in fact some connections which are worth highlighting.

The first and arguably most important remark is that all processes with bounded increments are sub-Gaussian, and hence sub-$\psi$. Indeed, whenever a process $(Y_t)_{t\in\N}$ satisfies $\Delta Y_{t} = Y_{t+1}-Y_{t}\in[a,b]$, it follows by Hoeffding's lemma (see e.g.~\cite{hoeffding1994probability,hertz2020improved}) that
\begin{equation} \label{eq:bounded-rvs-are-subpsi}
	\log \E\left[ e^{ u\,(\Delta Y_t-\E[\Delta Y_t \mid \Fcal_t ]) } \mid \Fcal_t \right] 
% 	\leq - u \, (b-a) \,p + \ln(1 - p + p\, e^{u \, (b-a)})
	\leq \frac{1}{8} u^2 \, (b-a)^2
	\;,
\end{equation}
% where $p=\nicefrac{(\E[Y]-a)}{(b-a)}$, and
where the second and third expression in the above inequality are the cumulant generating functions of a Gaussian random variable. Hence, we have that processes with bounded increments are sub-$\psi$, where $\psi$ is given by either expressions in~\eqref{eq:bounded-rvs-are-subpsi}. 

There exists a deeper connection between the two as follows. Let $Z = (Z_t)_{t\in\N}$ be a supermartingale difference process and let $Y_t = \sum_{i=1}^t Z_i$. Define for each $n\in\N$ and $u\geq0$ the family of processes
%\begin{equation*}
$	M_t^{u,n} = \prod_{i=1}^t \left( 1 + u Z_i/n \right)^n \, e^{-\mu_n(u)}$,
%	\;,
%\end{equation*}
where the collection of compensators $\{\mu_n\}_{n\in\N}$ with $\mu_n:[0,\infty)\to [0,\infty]$ satisfy %\begin{equation*}
$\log \E\left[ \left( 1 + u Z_i/n \right)^n \mid \Fcal_{i-1} \right] \leq \mu_i(u)$
%\end{equation*}
for all $i\in\N$ and $u>0$. Since the function $z \mapsto \left( 1 + z/n \right)^n$ is monotone, nonnegative and convex function on $z\in[-n,\infty)$, we find that whenever $u \, Z_i \geq -n$ for all $i\in\N$, the process $L^{u,n}$ is a nonnegative supermartingale with initial value $L_0^{u,n}=1$.

Whenever $n=1$, we recover a set of processes with multiplicative increments that are linear in the $Z_i$,
%\begin{equation*}
$	M_t^{u,1} = \prod_{i=1}^t \left( 1 + {u \, Z_i} \right) \, e^{-\mu_1(u)}$.
%	\;.
%\end{equation*}
Noting that if $u Z_i > -1$ and $\E\left[Z_i \mid \Fcal_{i-1}\right]\leq 0$, we find that $\mu_1 \equiv 0$ is a valid compensator and $M^{u,1}$, defined in this manner, matches the structure of the uniformly bounded test supermartingales presented in Section~\ref{sec:uniformly-bounded-hypotheses}.
On the other hand, we note that as $n\to\infty$,  $\left( 1 + z/n \right)^n \to e^z$ pointwise, which yields the process
%\begin{equation*}
$	M_t^{u,\infty} = \prod_{i=1}^t e^{u \, Z_i - \mu_{\infty}(u)}$,
%	\;,
%\end{equation*}
revealing a similar structure to the test supermartingales discussed in Section~\ref{sec:sub-psi-hypotheses}. Indeed, whenever the process $Y_i$ is sub-$\psi$ and $u\geq 0$, we find that $L^{u,\infty}$ is a nonnegative supermartingale.
\end{remark}

The families of test supermartingales presented in Lemmas~\ref{lem:test-smg-family-bded-ident} and~\ref{lem:test-smg-family-subpsi-ident} can be thought of as generalizations of the processes presented in~\cite{wau_ram_20} for the purpose of building confidence sequences for means of bounded random variables. In particular, let us consider the functional $T(\mu) = \E_\mu[X]$ with identification function $m(\lambda,x)=\lambda - x$, where the random process $\{X_i\}_{i\in\N}$ is constrained to the interval $[0,1]$. Applying Lemma~\ref{lem:test-smg-family-bded-ident} to this particular setting, we recover the capital process in~\citet[\S4]{wau_ram_20}. Similarly, as noted in Remark~\ref{rem:connection-bounded-and-subpsi}, the process $\sum_{i=1}^t \langle z , m(X_i,\lambda) \rangle$ is sub-$\psi$ for all $\lambda\in [0,1]$ with $\psi_{\lambda}(u) = \frac{1}{8} u^2$ and $V_t = t$, allowing us to recover via Lemma~\ref{lem:test-smg-family-subpsi-ident} the so-called Chernoff and predictably-mixed Chernoff martingales~\citep[\S2.3 \& \S3.1]{wau_ram_20}.

\section{Power via online convex optimization} 
\label{sec:power-via-OCO}

Consider a null hypothesis $\Hnull$ of the form \eqref{eq_sup_mg_hypothesis} given in terms of a family of base test supermartingales $\{ L^\theta \}_{\theta \in \Theta}$. We assume that each $L^\theta$ is of product form,
\begin{equation} \label{eq_L_theta_product_form}
L^\theta_t = \prod_{i=1}^t f_\theta(X_i)
\end{equation}
for some nonnegative function $f_\theta(x)$. This is the case for all the hypotheses considered in Section~\ref{sec_null_hypothesis}. We now address the problem of designing a powerful test supermartingale $W$. As a starting point we take the GRO criterion in \eqref{eq_GROW} and search among processes obtained by predictable mixing as in \eqref{eq:predictably-controlled-test-smg-1}. We must choose a distribution $\hat Q \in \Mcal(\Xcal^\N)$ to bet on. This distribution will not be fully specified ahead of time, but rather learnt adaptively. For $t \in \N$ we set $\hat Q(X_{t+1} \in \cdot \mid \Fcal_t) = (1/t) \sum_{i=1}^t \delta_{X_i}$,
the empirical measure of the data observed so far. Together with $\hat Q(X_1 \in \cdot)$, which we choose arbitrarily, this uniquely specifies the distribution $\hat Q$. The GRO problem \eqref{eq_GROW} at time $T$, restricted to predictable mixtures $V$ as in \eqref{eq:predictably-controlled-test-smg-1}, now takes the form
\begin{equation} \label{eq_GROW_OCO}
\maximize_{\pi_{T+1}} \ \sum_{i=1}^T \log \int_\Theta f_\theta(X_i) \pi_{T+1}(d\theta).
\end{equation}

Let us give some further details for how to get from \eqref{eq_GROW} to \eqref{eq_GROW_OCO}. First of all, the supermartingale constraint in \eqref{eq_GROW} is automatically satisfied since $V$ is a predictable mixture. Next, the form of $W$ and $L^\theta$ imply that $W_{T+1} / W_T = \int_\Theta (L^\theta_{T+1} / L^\theta_T) \pi_{T+1}(d\theta) = \int_\Theta f_\theta(X_T) \pi_{T+1}(d\theta)$. Finally, because $\hat Q(X_{T+1} \in \cdot \mid \Fcal_T)$ is defined as the empirical measure of $X_1,\ldots,X_T$ it follows that the objective function in \eqref{eq_GROW} is equal to
\[
\E_{\hat Q}\left[ \log \frac{W_{T+1}}{W_T} \mid \Fcal_T \right] = \E_{\hat Q}\left[ \log \int_\Theta f_\theta(X_{T+1}) \pi_{T+1}(d\theta) \mid \Fcal_T \right] = \frac{1}{T} \sum_{i=1}^T \log \int_\Theta f_\theta(X_i) \pi_{T+1}(d\theta).
\]
Since the factor $1/T$ does not affect the optimization, we arrive at \eqref{eq_GROW_OCO}.

Although the maximization problem \eqref{eq_GROW_OCO} is concave in $\pi_{T+1}$, for practical reasons we wish to avoid optimizing over this potentially infinite dimensional quantity. Instead, as discussed in Subsection~\ref{sec:betting-against-the-null}, we restrict the optimization to the much smaller set of `one-asset strategies' of the form $\pi_{T+1} = \delta_{\theta_{T+1}}$ for some $\theta_{T+1} \in \Theta$ that may depend on $X_1,\ldots,X_T$. Doing so simplifies \eqref{eq_GROW_OCO} further to
\begin{equation} \label{eq_GROW_OCO_2}
\maximize_{\theta_{T+1}} \ \sum_{i=1}^T \log f_{\theta_{T+1}}(X_i).
\end{equation}
Due to the product form of $L^\theta$ the objective function is actually of the even simpler form $\log  L^{\theta_{T+1}}_T$. Here the dependence on the optimization variable $\theta_{T+1}$ is no longer concave in general. Nonetheless, the following lemma shows that in a wide range of examples concavity does, in fact, hold. This will allow us to apply results from Online Convex Optimization, either directly to \eqref{eq_GROW_OCO_2} itself, or to regularized versions of it.

\begin{lemma}
	\label{lem:concavity-test-smg}
	Let $\{L^\theta\}_{\theta\in\Theta}$ satisfy one of the following conditions.
	\begin{enumerate}
		\item There exists a convex set $\Theta\subseteq\Lambda$ such that $\{L^\theta\}_{\theta\in\Theta}$ are test supermartingales for a bounded elicitable hypothesis defined according to Lemma~\ref{lem:test-smg-family-bded-elic}, where the scoring function $\Theta \ni \lambda \mapsto s(\lambda,x)$ is convex for all $x\in\Xcal$.
		\item There exists a convex set $\Theta\subseteq \Abd_{m,\lambda_0}$ such that $\{L^\theta\}_{\theta\in\Theta}$ are test martingales for a bounded identifiable hypothesis defined according to Lemma~\ref{lem:test-smg-family-bded-ident}.
		\item There exists a finite dimensional convex set $\Theta$ and a map 
		\[
		\Theta\ni \theta \mapsto (\lambda(\theta),u(\theta)) \subseteq \Lambda \times [0,\umax)
		\]
		such that the collection $\{L^{\theta}\}_{\theta\in\Theta}=\{ L^{\lambda(\theta),u(\theta)}\}_{\theta\in\Theta}$ are test supermartingales for a sub-$\psi$ elicitable hypothesis defined according to Lemma~\ref{lem:test-smg-family-subpsi-elic} and 
		\[
		\Theta \ni \theta \mapsto u(\theta) \,( s(\lambda_0,x_i) - s(\lambda(\theta),x_i)) - v_i^{\lambda(\theta)} \psi_{\lambda(\theta)}(u(\theta))
		\]
		is almost surely concave.
		\item \label{lem:concavity-test-smg:subpsi-ident}
		There exists a finite-dimensional convex set $\Theta$ and a map 
		\[
		\Theta\ni \theta \mapsto (\eta(\theta),u(\theta)) \subseteq \Apsi_{m,\lambda_0} \times [0,\umax)
		\]
		such that
		$\{L^{\theta}\}_{\theta\in\Theta} = \{ L^{\eta(\theta),u(\theta)}\}_{\theta\in\Theta}$ are test supermartingales for a sub-$\psi$ identifiable hypothesis defined according to Lemma~\ref{lem:test-smg-family-subpsi-ident}, and 
		\[ 
		\Theta \mapsto u(\theta) \, \langle \eta(\theta) \mathrel{,} m(\lambda_0,x_i) \rangle - v_i^{\lambda(\theta)} \psi_{\eta(\theta)}(u(\theta))
		\]
		is almost surely concave.
	\end{enumerate}
	Then for each $t\in\N$, the map $\theta \mapsto \log L_t^\theta$ is concave a.s.
\end{lemma}

The proof of the lemma uses that the composition of an increasing concave function with a concave function is concave and is given in the appendix.

% \begin{remark}
% We note here that Lemma~\ref{lem:concavity-test-smg} shows that test supermartingales generated by identifiable hypotheses are always guaranteed to yield concave log-return maps, whereas elicitable hypotheses require additional assumptions so that the scoring function is convex.
% \end{remark}

\subsection{Regret and asymptotic power}

By repeatedly solving \eqref{eq_GROW_OCO_2} (or a regularized version of it) in each time period, we obtain a predictable sequence $(\theta_t)_{t \in \N}$ which produces the test supermartingale
\begin{equation} \label{eq:predictably-controlled-test-smg-singleasset}
W_t = \prod_{i = 1}^t f_{\theta_i}(X_i).
\end{equation}
A standard way of measuring the quality of the sequence $(\theta_t)_{t \in \N}$ is the \emph{regret}, defined at each time $t$ by
\begin{equation} \label{eq:def-regret}
	\reg_t := 
	\max_{\theta \in \Theta} \left\{ \log L_t^\theta - \log W_t \right\}.
\end{equation}
The regret represents the difference between the log-value of the best retrospectively chosen single-asset portfolio, $\max_{\theta\in\Theta} \log L_t^\theta$, and the given log-wealth $\log W_t$. Various well-known algorithms for solving either \eqref{eq_GROW_OCO_2} or regularized versions of it, yield regret that grows sublinearly,
\[
\reg_T = o(T) \,,
\]
where we emphasize that this holds almost surely, that is,  $\lim_{T\to\infty} \reg_T/T = 0$ a.s.
We review some of these algorithms in Subsection~\ref{sec:oco-algorithms}. The following theorem shows that if regret grows sublinearly, then the test \eqref{eq:test-smg-test} constructed from the test supermartingale $W$ has asymptotic power one.

%Intuitively, any algorithm producing a sequence of $(\theta_i)_{i\in\N}$ such that $\reg_T = o(T)$ is asymptotically performing well, since its average regret per iteration $\nicefrac{\reg_T}{T}$, vanishes in the long-run. This implies that the average log-growth of $V_t$ is asymptotically equal to that of the best retrospectively chosen single-asset portfolio.
%
%
%
%Now, let us assume that an algorithm generates a sequence of controls $\{\theta_t\}_{t\in\N}$ aimed at optimizing the long-term regret of the composite test supermartingale $V$ given by equation~\eqref{eq:predictably-controlled-test-smg-singleasset}. Moreover, let us assume that $\{\theta_t\}_{t\in\N}$ achieves a regret bound of $\reg_T \leq B_T$, where we assume that the regret bound satisfies $B_0=0$, and $B_T=o(T)$. Using the definition of regret in equation~\eqref{eq:def-regret}, we have that for all $\theta^\prime \in \Theta$,
%\begin{equation} \label{eq:regret-test-smg-inequality}
%	\frac{1}{T} \log V_T 
%	\geq
%	\frac{1}{T} \log L_T^{\theta^\prime}
%	- \frac{1}{T} B_T
%	\;,
%\end{equation}
%where right-most term, $\nicefrac{B_T}{T}$, is guaranteed to vanish in the limit by assumption. Using this lower bound on the value process, we arrive at the following theorem.

\begin{theorem}[Sublinear regret implies asymptotic power]
	\label{thm:asymptotic-power-1}
	Let $W=(W_t)_{t\in\N}$ be a predictably mixed test supermartingale process defined according to~\eqref{eq:predictably-controlled-test-smg-singleasset}, generated by a family of test supermartingales $\{L^\theta\}_{\theta\in\Theta}$ and a sequence $(\theta_t)_{t\in\N}$ achieving $\reg_T = o(T)$. Let $\alpha \in (0,1]$, and let $\tau$ be the test induced by $W$, that is, $\tau = \inf\{ t\in\N \colon W_t > \alpha^{-1} \}$.  Suppose that for a probability measure $Q\notin\Hnull$, there exists $\theta^\prime \in \Theta$ for which 
	\begin{equation}\label{eq:liminf_condition}
	    Q\left( \liminf_{T\to\infty}\frac{\log L_T^{\theta^\prime}}{T} > 0 \right) = 1,
	\end{equation}
	then
	\begin{equation*}
		Q\left(\text{$\tau$ rejects $\Hnull$}\right) = Q(\tau < \infty)= 
		Q\left( 
		\lim_{T\to\infty} \max_{0\leq t\leq T} W_t > \alpha^{-1}
		\right) = 1,
	\end{equation*}
	that is, $\tau$ eventually rejects $\Hnull$ with probability one.
\end{theorem}
\begin{proof}
Using the definition of regret in equation~\eqref{eq:def-regret}, we have that for all $\theta \in \Theta$,
\[
	\frac{1}{T} \log W_T 
	\geq
	\frac{1}{T} \log L_T^{\theta}
	- \frac{1}{T} B_T
	\;,
\]
where $B_T = \reg_T$. Applying the assumptions that $B_T = o(T)$ and \eqref{eq:liminf_condition}, we have that 
	\begin{align*}
		\liminf_{T\to\infty} \frac{ \log \max_{0\leq t\leq T} W_t}{T} \geq
		\liminf_{T\to \infty} \frac{\log W_T}{T}
		&\geq 
		\liminf_{T\to \infty} \left\{ \frac{\log L_T^{\theta^\prime}}{T} - \frac{B_T}{T} \right\}
		=
		\liminf_{T\to \infty} \frac{\log L_T^{\theta^\prime}}{T}
		> 0\;, 
	\end{align*}
	where these inequalities hold $Q$-almost surely. %Now, since $\lim_{T\to\infty} \frac{ \max_{0\leq t\leq T} \log V_t}{T} > 0$ $Q$-almost surely, by the definition of the limit 
	This implies that, $Q$-almost surely, there exists a $T^\prime\in\N$ and $\epsilon>0$ such that $\max_{0\leq t\leq T} W_t \geq e^{\epsilon \, T}$ for all $T>T^\prime$. Hence, we may conclude that $\lim_{T\to\infty} \max_{0\leq t\leq T} W_t \geq \lim_{T \to \infty} e^{\epsilon T} = \infty$ $Q$-almost surely, as desired.
\end{proof}

% \begin{remark}
%     The proof of Theorem~\ref{thm:asymptotic-power-1} provides us with a lower bound on the number of samples required to reject $\Hcal_0$ in the worst case. In particular, if we define $\bar{\ell} = \sup_{\theta\in\Theta} \liminf_{T\to \infty} \log L_T^\theta/T$, we obtain
%     \begin{equation*}
%         \tau \geq \inf \left\{
%             T\in\N^+ \mathrel{\colon}
%             T \bar{\ell} - \reg_T > - \log \alpha
%         \right\}
%     \end{equation*}
%     as the minimum number of samples required to reject $\Hcal_0$, where this quantity is finite whenever the conditions of Theorem~\ref{thm:asymptotic-power-1} are met.
% \todoPC{Give a simple example in which we can compute $\bar{\ell}$ or verify to make sure is true}
% \end{remark}

Each of the algorithms presented in Subsection~\ref{sec:oco-algorithms} achieve sublinear regret, as required in the statement of Theorem~\ref{thm:asymptotic-power-1}. The second assumption in Theorem \ref{thm:asymptotic-power-1} that must be met to achieve asymptotic power one, is the existence of a $\theta^\prime$ that satisfies \eqref{eq:liminf_condition}. In the next lemma, we show that there are easily verifiable sufficient conditions to guarantee this condition when data is generated by a stationary ergodic process.

\begin{proposition} \label{prop:stationary-ergodic-asymtotic-power-1}
	Suppose that for $Q\notin\Hcal_0$, the data-generating process $(X_t)_{t\in\N}$ is stationary and ergodic with invariant measure $Q_{\infty}$. Let $\{L_t^\theta\}_{\theta\in\Theta}$ be a collection of $\Hnull$ test supermartingales where one of the following conditions holds.
	\begin{enumerate}
		\item $\Hnull$ is a uniformly bounded elicitable hypothesis with $\{L_t^\theta\}_{\theta\in\Theta}$ defined according to Lemma~\ref{lem:test-smg-family-bded-elic} where $\Theta\subseteq\Lambda$ and there exists $\lambda^\prime \in \Theta$ such that $\E_{Q_{\infty}} \log\left(1 + s(\lambda_0,X_{\infty}) - s(\lambda^\prime,X_{\infty}) \right) > 0$.
		\item $\Hnull$ is a uniformly bounded identifiable hypothesis with $\{L_t^\theta\}_{\theta\in\Theta}$ defined according to Lemma~\ref{lem:test-smg-family-bded-ident} where $\Theta\subseteq \Abd_{m,\lambda_0}$ and there exists $\eta^\prime\in\Theta$ such that $\E_{Q_{\infty}} \log\left(1 + \langle \eta^\prime , m(\lambda_0,X_{\infty}) \rangle \right) > 0$.
		\item $\Hnull$ is a sub-$\psi$ elicitable hypothesis with $\{L_t^\theta\}_{\theta\in\Theta}$ defined according to Lemma~\ref{lem:test-smg-family-subpsi-elic} where $\Theta = \Theta^\prime \times [0,\epsilon)$ in which $0<\epsilon\leq\umax$, $\Theta^\prime \subseteq \Lambda$, $\lim_{t\to\infty} \tfrac{1}{t} \sum_{i=1}^t v_i^\lambda < \infty$ a.s. for all $\lambda\in\Lambda$, and there exists $\lambda^\prime \in \Theta^\prime$ such that $\E_{Q_{\infty}} [s(\lambda_0,X_{\infty}) - s(\lambda^\prime,X_{\infty})] > 0$.
		\item $\Hnull$ is a sub-$\psi$ identifiable hypothesis with $\{L_t^\theta\}_{\theta\in\Theta}$ defined according to Lemma~\ref{lem:test-smg-family-subpsi-ident} where $\Theta = \Theta^\prime \times [0,\epsilon)$ in which $0<\epsilon\leq\umax$, $\Theta^\prime \subseteq \Apsi_{m,\lambda_0}$, $\lim_{t\to\infty} \tfrac{1}{t} \sum_{i=1}^t v_i^\eta < \infty$ a.s. for all $\eta\in\Apsi_{m,\lambda_0}$ , and there exists $\eta \in\Theta^\prime$ such that $\E_{Q_{\infty}} \langle \eta \mathrel{,} m(\lambda_0,X_{\infty}) \rangle > 0$.
	\end{enumerate}
	Then there exists $\theta^\prime \in \Theta$ for which \eqref{eq:liminf_condition} holds.
\end{proposition}
\begin{proof}
%	We prove the claim by treating cases (i)-(ii), (iii) and (iv) separately.
	
	\emph{Cases (i) and (ii):} By Lemma~\ref{lem:test-smg-family-bded-elic} or~\ref{lem:test-smg-family-bded-ident}, we may write each test submartingale as
%	\begin{equation*}
$		\log L_t^\theta = \sum_{i=1}^t \log\left(1 + f(\theta,X_i) \right)$,
%	\end{equation*}
	where we have either $f(\theta,x) = s(\lambda_0,x) - s(\theta,x)$ or $f(\theta,x) = \langle \theta,m(\lambda_0,x)\rangle$, respectively. Since $(X_t)_{t\in\N}$ is assumed to be stationary and ergodic under $Q_{\infty}$, we may apply the Birkhoff-Khintchin ergodic theorem~\cite[Theorem 1]{cornfeld2012ergodic} to obtain
	\begin{equation}
		\lim_{t\to\infty} \frac{\log L_t^\theta}{t}
		= \lim_{t\to\infty} \frac{1}{t} \sum_{i=1}^t \log\left(1 + f(\theta,X_i) \right)
		= \E_{Q_{\infty}}\left[ \log\left(1 + f(\theta,X_{\infty}) \right) \right]
		\;.
	\end{equation}
%	where there exists a $\theta\in\Theta$ where the right-most quantity is strictly positive by assumption in (i) and (ii), concluding the proof.
	
\noindent\emph{Cases (iii) and (iv):} By Lemma~\ref{lem:test-smg-family-subpsi-elic} or~\ref{lem:test-smg-family-subpsi-ident}, we may write each test submartingale as
	\begin{equation*}
		\log L_t^\theta = \log L_t^{(\theta^\prime,u)} = \sum_{i=1}^t \left\{ 
		    u \, f(\theta^\prime,X_i) - v_i^{\theta^\prime} \psi_{\theta^\prime}(u)
		\right\}
		\;,
	\end{equation*}
	where have $u\in[0,\epsilon)$ with $\epsilon>0$, $v_i^{\theta^\prime}>0$ and,  either $f(\theta^\prime,x) = s(\lambda_0,x) - s(\theta',x)$ in the elicitable case, or $f(\theta^\prime,x) = \langle \theta' , m(\lambda_0,x)\rangle$ in the identifiable case.
%	\begin{enumerate}[label=\roman*.]
%		\item \emph{elicitable case:} For $\theta^\prime=\lambda\in\Lambda$,
%		$f(\theta^\prime,x) = s(\lambda_0,x) - s(\lambda,x)$ and $\tilde\psi_{\theta^\prime}(u) = \psi_\lambda(u)$, or
%		\item \emph{identifiable case:} For $\eta=\theta^\prime\in \Theta^\prime \subseteq \Abd_{m,\lambda_0}$, $f(\theta^\prime,x) = \langle \eta , m(\lambda_0,x)\rangle$ and $\tilde \psi_{\theta^\prime}(u) = \psi_{\eta}(u)$. 
%	\end{enumerate}
	By assumption, in each case there exists a $\theta^\prime_0$ (where $\Theta^\prime \subseteq \Lambda$ in the elicitable case and $\Theta^\prime \subseteq \Apsi_{m,\lambda_0}$ in the identifiable case) such that $E_{Q_\infty} f(\theta^\prime_0,X_{\infty}) = c > 0$. 
	Since $(X_t)_{t\in\N}$ is assumed to be stationary and ergodic under $Q$, we may apply the Birkhoff-Khintchin ergodic theorem~\cite[Theorem 1]{cornfeld2012ergodic} to $L^{(\theta^\prime_0,u)}$ to obtain
	\begin{align*}
		\lim_{t\to\infty} \frac{\log L_t^{(\theta^\prime_0,u)}}{t}
		&= \lim_{t\to\infty} \frac{1}{t} \sum_{i=1}^t \left\{ u\, f(\theta^\prime_0,X_i) -  v_i^{\theta_0^\prime} \psi_{\theta^\prime_0}( u ) \right\}
		\\ &= u\, \E_{Q_{\infty}}\left[f(\theta^\prime_0,X_{\infty})\right] -  \psi_{\theta^\prime_0}(u)
		= u\, c - \bar{v} \psi_{\theta^\prime_0}(u)
		\;,
	\end{align*}
	where $0< \bar{v} = \lim_{t\to\infty} \frac{1}{t} \sum_{i=1}^t \bar{v}$ which is non-negative and finite by assumption.
	By Definition~\ref{def:sub-psi-rv}, we find that $\psi_{\theta^\prime_0}(0) = 0$ and $\lim_{u\searrow 0} \psi_{\theta^\prime_0}(0)/u = 0$.
	Hence, letting $g(u) = u\, c - \bar{v} \psi_{\theta^\prime_0}(u)$, we find that, $g(0) = 0$ and $g^\prime(0) = \lim_{u\searrow0} g(u)/u = c > 0$. Hence, there exists $u^\prime_0\in(0,\epsilon)$ such that $g(u^\prime_0)>0$. By picking this $u^\prime_0$, we then have that
%	\begin{equation*}
		$\lim_{t\to\infty} \log L_t^{(\theta^\prime_0,u^\prime_0)}/t = g(u^\prime_0) = c > 0$,
%	\end{equation*}
	yielding the desired result.
\end{proof}

We remark here that two common cases in which Proposition~\ref{prop:stationary-ergodic-asymtotic-power-1} may apply are when $(X_t)_{t \in \N}$ is either i.i.d.\ or generated by a stationary, irreducible and aperiodic Markov chain. Moreover, conditions \emph{(iii)} and \emph{(iv)} of Proposition~\ref{prop:stationary-ergodic-asymtotic-power-1} hold for any $Q\notin\Hnull$, which can be seen by inspecting~\eqref{eq_elicitable_null} or~\eqref{eq_identifiable_null}. 

To conclude, Theorem~\ref{thm:asymptotic-power-1} can be interpreted at a high-level result stating that whenever there are alternative hypotheses which offer `sufficient evidence against the null', regret-optimal strategies will asymptotically reject the null with probability one.  Proposition~\ref{prop:stationary-ergodic-asymtotic-power-1} quantifies the notion of `sufficient evidence against the null' in the case of a stationary ergodic data generating process.

\subsection{Online Convex Optimization Algorithms} \label{sec:oco-algorithms}

We summarize a few simple but effective OCO algorithms that can be applied to optimizing the predictably mixed test supermartingale $V = (V_t)_{t\in\N}$ by selecting a sequence of single-asset portfolios given by $(\theta_t)_{t\in\N}$ which maximize the growth rate of~\eqref{eq:predictably-controlled-test-smg-singleasset}. We point interested readers to~\cite{hazan2016introduction,shalev2011online} for a comprehensive introduction to OCO.

For the remainder of the section, we assume that for all $t\in\N$, the maps of forward differences defined on a convex set $\Theta$,
\[
\theta \mapsto \Delta \log L_{t-1}^\theta 
% = \log\left(\frac{L_{t}^\theta}{L_{t-1}^\theta}\right) 
= \log L_{t}^\theta - \log L_{t-1}^\theta
\]
are concave. We say that a convex function $f:\Theta\rightarrow\R$ is strongly convex with parameter $\mu>0$ whenever $f(\theta_1) - f(\theta_0) - \langle \nu_0 , \theta_1 - \theta_0\rangle \geq \mu \| \theta_1 - \theta_0 \|^2$ for all $\theta_0,\theta_1\in\Theta$, where $\nu_0 \in \partial_{\theta} f(\theta_0)$ is an element the subgradient of $f$ at $\theta_0$. Similarly, $f$ is strongly concave whenever $-f$ is strongly convex. We define the norm of the subgradient of a convex function $f:\Theta\rightarrow\R$ as $\|\partial_\theta f(\theta)\| = \sup_{\nu\in\partial_\theta f(\theta)} \|\nu\|$.\\

\noindent\textbf{Follow The Leader.} We define the Follow The Leader (FTL) algorithm as choosing at each iteration $t+1$
\begin{equation} \label{eq:ftl-definition}
	\theta_{t+1} \in \argmin_{\theta\in\Theta} -\log L_t^\theta
	= \argmin_{\theta\in\Theta} -\sum_{1\leq i \leq t} \Delta \log  L_{i-1}^\theta
	\tag{FTL}
	\;.
\end{equation}
Hence, at each iteration, \eqref{eq:ftl-definition} picks $\theta\in\Theta$ such that $\log L_t^\theta$ has the largest average growth rate in hindsight. Implicitly, this algorithm assumes that computing the $\argmax$ at each round can be done relatively easily. Under the additional assumptions that (i) $\sup_{\theta\in\Theta}\| \partial_{\theta} (-\Delta \log L_{t-1}^\theta ) \| \leq G < \infty$ and  (ii) the $\Delta \log {L_{t}^\theta}$ are strongly concave, we have that $\reg_T \leq \frac{G^2}{2}(1 + \log T)$. We point the reader to~\citet[\S 3.6]{mcmahan2017survey} for a derivation of this bound. This algorithm is not recommended for the case when $\Delta \log L_t^\theta $ are not strongly concave. Indeed, it is possible to create counterexamples, such as those presented in \citet[Example 2.2]{shalev2011online} or~\citet[Page 65]{hazan2016introduction}, of non strongly-concave functions where the algorithm induces super-linear growth in the regret. The bounded gradient assumption (i) will be satisfied by, for example, Lipschitz continuous functions.\\

\noindent\textbf{Follow The Regularized Leader.} A simple fix for this problem leads to the second algorithm, which involves the inclusion of regularization, increasing the algorithm's stability. The most obvious implementation of this concept is the Follow The Regularized Leader (FTRL) algorithm, in which we introduce regularization terms to the optimization problem in~\eqref{eq:ftl-definition}. Precisely, at each iteration $t+1$, FTRL selects
\begin{equation} \label{eq:ftrl-definition}
	\theta_{t+1} \in \argmin_{\theta\in\Theta} 
	\left\{
	-\log L_t^\theta +
	\sum_{1\leq i \le t} r_{i-1}(\theta)
	\right\}
	\tag{FTRL}
	\;,
\end{equation}
where the $\{ r_t \}_{t\in\N}$ are a sequence of strongly convex functions for which either (a) $x_0 = \argmin_{\theta\in\Theta} r_t(\theta)$ for all $t$, in which the algorithm is the \emph{FTRL-Centered} variant or (b) for which we assume that $x_t = \argmin_{\theta\in\Theta} r_t(\theta)$ where this variant is named \emph{FTRL-Proximal}. Under the additional assumptions that (i) $\sup_{\theta\in\Theta}\|  \partial_{\theta} (-\Delta \log L_{t-1}^\theta  ) \| \leq G < \infty$ and that (ii) $\diam(\Theta) = \sup_{x,y\in\Theta} \|x-y\| < \infty$, one can devise a sequence of centered or proximal regularizers $\{r_t\}_{t \in \N}$ such that the algorithms enjoy the regret bound $\reg_T = O(\sqrt{T})$. We point the reader to~\citet[Section 3]{mcmahan2017survey} for a summary of the exact conditions and specific bounds. \\

\noindent\textbf{Online Gradient Descent.} The $\argmin$ expressions in~\eqref{eq:ftl-definition} and~\eqref{eq:ftrl-definition} may be difficult to compute directly introducing problems in the implementation of FTL or FTRL. Instead, using gradients collected in each step, one can follow the direction of steepest ascent of $\Delta \log L_{t-1}^\theta$ at each iteration. This produces the Online Gradient Descent (OGD) algorithm, where we choose at each iteration
\begin{equation} \label{eq:online-gradient-descent-def}
	\theta_{t+1} = \Pi_{\Theta}\left\{
	\theta_{t} - \eta_t \nu_t
	\right\}
	\text{ where }
	\nu_t \in \partial_\theta \left( - \Delta \log {L_{t-1}^\theta}\right)
	\tag{OGD}
	\;,
\end{equation}
where $\{\eta_t\}_{t\in\N}$ is a sequence of positive learning rates and $\Pi_{\Theta}(x) = \argmin_{y\in\Theta} \|y-x\|$ is the projection operator. This algorithm has the advantage that it is extremely simple to compute, provided that gradients are available at each step. For an appropriate choice of $\{\eta_t\}_{t\in\N}$, and under the additional assumption that (i) $\sup_{\theta\in\Theta}\|  \partial_{\theta} (- \Delta \log\left( L_{t-1}^\theta \right) ) \| \leq G < \infty$ and (ii) $\diam(\Theta)<\infty$, the algorithm has a regret bound of $\reg_T=O(\sqrt{T})$. We point the reader to~\cite{zinkevich2003online} and \citet[\S3.1]{mcmahan2017survey} for the specific bounds and conditions. 

We summarize the bounds and assumptions for the OCO algorithms in Table~\ref{tab:oco-alg-summary}, where we emphasize that the growth rate of the regret in the right-hand column holds almost surely regardless of the data-generating measure. Beyond the algorithms presented here, there exist a plethora of online optimization algorithms which may leverage the geometry of the index set $\Theta$ or past information about gradients in order to improve rates of convergence. For a broader survey of available OCO algorithms, we point the reader to \cite{hazan2016introduction,shalev2011online,mcmahan2017survey}. We also point out that although the algorithms here offer guarantees for convex loss functions, they can in principle be used on non-convex optimization problems as well. However, while still yielding valid tests, the worst-case guarantees may no longer hold.

%\begin{remark}
The use of OCO methods for building more powerful predictably mixed test supermartingales is related to various online methods presented in~\cite{wau_ram_20} for confidence sequence building for means of bounded random variables. In particular, Kelly betting~\cite[\S5.2,5.3,5.6]{wau_ram_20} can be interpreted as a variation of the \ref{eq:ftl-definition}~algorithm to a family of processes of the form of Lemma~\ref{lem:test-smg-family-bded-ident} for bounded identifiable functionals. Similarly, the Online Newton Step Algorithm discussed in~\citet[\S5.5]{wau_ram_20} can be thought of another method in the family of OCO algorithms. %applied to families of processes such as those in Lemma~\ref{lem:test-smg-family-bded-ident}. 
The authors of \cite{wau_ram_20} acknowledge and elaborate on the connection between their confidence sequence building methods with OCO and coin-betting algorithms in~\citet[Appendix~D]{wau_ram_20} but they do not use regret bounds to derive asymptotic power guarantees.
%\end{remark}

%\begin{remark}
 Although largely ignored over the course of this section, another important aspect to consider when selecting an algorithm for online convex optimization is the computational complexity of the algorithm. For example, it is typically the case that~\eqref{eq:ftl-definition},~\eqref{eq:ftrl-definition} or other loss-minimization based algorithms will be run slower than gradient-based algorithms such as~\eqref{eq:online-gradient-descent-def}. On the other hand, although~\eqref{eq:online-gradient-descent-def} and~\eqref{eq:ftrl-definition} have similar asymptotic performance, we typically find that~\eqref{eq:ftrl-definition} may perform better on average, which can be reflected in the constants associated with their regret bounds, see for example \cite{mcmahan2017survey}. 
%\end{remark}

% \begin{table}[ht]
% \centering
% \begin{tabular}{c|l|l|l}
% \cline{2-3}
% \multicolumn{1}{l|}{} & \multicolumn{2}{c|}{\textbf{Assumptions}} &  \\ \hline
% \multicolumn{1}{|c|}{\textbf{Algorithm}} & \multicolumn{1}{c|}{$-\Delta \log V_i(\theta)$} & \multicolumn{1}{c|}{$\text{diam}(\Theta)$} & \multicolumn{1}{l|}{$\reg_T$} \\ \hline
% \multicolumn{1}{|c|}{\ref{eq:ftl-definition}} & strongly-convex + bounded gradients & any & \multicolumn{1}{l|}{$\Ocal(\log T)$} \\ \hline
% \multicolumn{1}{|c|}{\ref{eq:ftrl-definition}} & bounded gradients & $\leq R$ & \multicolumn{1}{l|}{$\Ocal(\sqrt{T})$} \\ \hline
% \multicolumn{1}{|c|}{\ref{eq:online-gradient-descent-def}} & bounded gradients & $\ leq R$ & \multicolumn{1}{l|}{$\Ocal(\sqrt{T})$} \\ \hline
% \end{tabular}
% \caption{Summary of assumptions and regret bounds for the OCO algorithms presented in Section~\ref{sec:oco-algorithms}}
% \label{tab:oco-alg-summary}
% \end{table}

\begin{table}[]
\centering
\begin{tabular}{@{}clll@{}}
\cmidrule(lr){2-3}
\multicolumn{1}{l}{} & \multicolumn{2}{c}{\textbf{Assumptions}} &  \\ \midrule
\textbf{Algorithm} & \multicolumn{1}{c}{$-\Delta \log W_t(\theta)$} & \multicolumn{1}{c}{$\text{diam}(\Theta)$} & $\reg_T$ \\ \midrule
\textbf{\ref{eq:ftl-definition}} & strongly-convex + bounded gradients & any & $O(\log T)$ \\
\textbf{\ref{eq:ftrl-definition}} & bounded gradients & $\leq R$ & $O(\sqrt{T})$ \\
\textbf{\ref{eq:online-gradient-descent-def}} & bounded gradients & $\leq R$ & $O(\sqrt{T})$
\end{tabular}
\caption{Summary of assumptions and regret bounds for the OCO algorithms presented in Section~\ref{sec:oco-algorithms}.}
\label{tab:oco-alg-summary}
\end{table}

%\subsection{Generalizations and Predictive Measures}

Although the OCO algorithms presented in Section~\ref{sec:oco-algorithms} offer regret guarantees which translate into asymptotic power, it is possible to generalize these methods in order to further improve performance in special cases. The algorithms~\ref{eq:ftl-definition},~\ref{eq:ftrl-definition} and~\ref{eq:online-gradient-descent-def}, attempt to maximize some version of the data-generated objective function
\begin{equation} \label{eq:empirical-log-return-objective}
	 \theta \mapsto \frac{1}{t}\log L_t^\theta	 = 
	 \frac{1}{t}\sum_{i=1}^t \Delta \log L^\theta(X_i)
	 = \E_{\hat P_t}\left[ \Delta \log L^\theta(X) \right]
	 \;,
\end{equation} 
where we write $\Delta \log L_i^\theta = \Delta \log L^\theta(X_i)$ as an explicit function of the data point $X_i$, and where $\E_{\hat P_t}$ represents the expected value with respect to the empirical measure $\hat P_t = \frac{1}{t}\sum_{i=1}^t \delta(X-X_i)$. 

If additional distributional information is known about $\{X_t\}_{t\in\N}$, one could replace the empirical distribution $\hat P_t$ with a predictive measure $\tilde P_t$ which may better represent the true data distribution. Modifying~\eqref{eq:ftl-definition} with this new measure, one obtains the algorithm
\begin{equation}
	\theta_{t+1} \in \argmin_{\theta\in\Theta} \E_{\tilde P_t}\left[ \log \Delta L^\theta(X) \right]
	\tag{FTLP}
	\;,
\end{equation}
and the same principle can be applied to modify both, \eqref{eq:ftrl-definition} and~\eqref{eq:online-gradient-descent-def}. The fact that we have complete freedom in choosing $\tilde P_{t}$ at each step can be a major advantage of the betting approach to sequential testing, as has been discussed in detail elsewhere; see \cite{MR4364897,wau_ram_20}.

\section{Confidence Sequences and Inverting Tests}
\label{sec:confidence-sequences}

Our focus has been on methodology for testing elicitable and identifiable hypotheses in a sequential setting. However, the techniques we have developed allow for the construction of anytime-valid confidence sequences. Given a functional $T(P):\Mcal_1(\Xcal)\to\Lambda\subseteq \R^d$ define an \emph{anytime-valid confidence sequence} $\{C_t\}_{t\in\N}$ \emph{at level} $\alpha\in(0,1)$ as a sequence $(C_t)_{t \in \N}$ of confidence sets $C_t\subseteq\Lambda$ for all $t\in\N$, satisfying the property that
\begin{equation} \label{eq:anytime-valid-confidence-sequence}
	\sup_{P\in\Hnull(\lambda_0)} P\left( \forall t\in\N \colon \lambda_0 \notin C_t \right) \leq \alpha
	\;,
\end{equation}
for any $\lambda_0\in\Lambda$. Here, $\Hnull(\lambda_0)$ denotes the null hypothesis at \eqref{eq_null_seq_T} where we now emphasize the dependence on $\lambda_0$ in the notation.
Equation~\eqref{eq:anytime-valid-confidence-sequence} can be interpreted as guaranteeing that, with high probability, the true value of the functional, $\lambda_0 \in T(P)$ is contained within, not only a single confidence set, but the entire sequence of confidence sets at once.

Using the anytime-valid hypothesis testing methodology, we can construct such confidence sequences for the
%types of 
elicitable and identifiable %hypotheses previously defined. %Indeed, for each $\lambda_0\in\Lambda$ define the map $\Lambda \ni \lambda_0 \mapsto \Hnull(\lambda_0)$, where either
%\begin{equation*}
%\Hnull(\lambda_0) = \left\{ P \in \Mcal_1(\Xcal^\N)\colon \sum_{i=1}^t \left( s(\lambda_0, x_i) - s(\lambda, x_i) \right) \text{ is a $P$-supermartingale for all } \lambda \in \Lambda \right\}
%\end{equation*}
%in the case where $T$ is elicitable, %or
%\begin{equation*}
%\Hnull(\lambda_0) = \left\{ P \in \Mcal_1(\Xcal^\N)\colon \sum_{i=1}^t m(\lambda_0, x_i) \text{ is a $P$-martingale} \right\}
%\end{equation*}
%in the case where $T$ is identifiable. We note that these are maps to 
martingale hypotheses of the form~\eqref{eq_elicitable_null} and~\eqref{eq_identifiable_null}. %which were the focus of the testing setting, but where the main change lies in that they are now parametric.
Assume that for each $\lambda_0\in\Lambda$ there exists an $\Hnull(\lambda_0)$ test supermartingale denoted by $V^{\lambda_0} = (W_t^{\lambda_0})_{t\in\N}$. Each $W^{\lambda_0}$ induces an anytime-valid test for $\Hnull(\lambda_0)$, so we may construct an anytime-valid confidence sequence $C=(C_t)_{t\in\N}$ by `inverting' these tests as follows. For each $t\in\N$, we define
\begin{equation} \label{eq:anytime-valid-ci-definition}
	C_t = \left\{ \lambda\in\Lambda \colon \max_{0\leq i \leq t} V_i^{\lambda} \leq \alpha^{-1} \right\}
	\;,
\end{equation}
the set of $\lambda$ whose associated null hypotheses have not yet been rejected by the tests induced by the associated $W^{\lambda_0}$. This construction produces an anytime-valid confidence sequence since
%%\begin{align*}
$	\sup_{P\in\Hnull(\lambda_1)}
	P(\exists t \, \text{ s.t.\ } \colon \lambda_1 \notin C_t )
	=
	\sup_{P\in\Hnull(\lambda_1)}
	P( \max_{t\in\N} W_t^{\lambda_1} > \alpha^{-1} ) \leq \alpha$.
%	\;.
%\end{align*}

We may apply the techniques from previous sections to construct confidence sequences for elicitable and identifiable functionals as follows. Depending on whether the functional $T$ in question is elicitable or identifiable and on whether it is \emph{uniformly bounded} or \emph{sub-$\psi$}, for each $\lambda_0\in\Lambda_0$ we construct a family of test supermartingales according to Lemmas~\ref{lem:test-smg-family-bded-elic},~\ref{lem:test-smg-family-bded-ident},~\ref{lem:test-smg-family-subpsi-elic} or~\ref{lem:test-smg-family-subpsi-ident}, which we use to construct a mixture test supermartingale $W^{\lambda_0}$ according to Lemma~\ref{lem:predictably-controlled-test-smg}. In Appendix~\ref{sec:one-sided-set-valued-hypotheses}, we show how this methodology can be extended for the purpose of testing one-sided and set-valued hypotheses.
% Using these processes, we construct the confidence sequence according to equation~\eqref{eq:anytime-valid-ci-definition}.

\section{Numerical Examples} \label{sec:simulation_study}

%\subsection{Bounded, Independent Data}

We begin by applying our tests to sequences of independent and identically distributed data.
Let $\{X_t\}_{t\in\N} \sim \text{ i.i.d. } \text{Beta}(\alpha,\beta)$, where $\alpha=2$ and $\beta=5$ are the parameters of the beta distribution. %We note here that this data generating process is bounded, which will allow us to apply tests for uniformly bounded elicitable and identifiable hypotheses. 

Our first experiment will be to test the mean and standard deviation of this data generating process, simultaneously. That is, we consider a functional $T:\Mcal_1(\Xcal)\to \R^2$ where
%\begin{equation}
%    \label{eq:mean-sd-functional}
$	T(Q) = %\begin{pmatrix}
		(\E_Q[ X_t ],
		\sqrt{\V_Q[X_t]})'$.
%	\end{pmatrix}.
%\end{equation}
Under the assumed ground truth data generating measure $P_{\text{GT}}$, the value of this functional is approximately $T(P_{\text{GT}})\approx\left( \begin{smallmatrix} 0.3 \\ 0.16 \end{smallmatrix}\right)$. First, we are interested in testing the null hypothesis
\begin{equation}
    \label{eq:mean-sd-hypothesis}
	\Hnull = \left\{
		P \in \Mcal_1(\Xcal^\N)\,:\,
		T(P(X_t \in \cdot \mid \Fcal_{t-1}))=\left( \begin{smallmatrix} 0.4 \\ 0.4 \end{smallmatrix}\right) \text{ for all } t \in \N, \ P\text{-a.s.}
	\right\}.
\end{equation}
Although this functional is both an elicitable and identifiable functional, we choose to test is as the latter, where we use the identification function
%\begin{equation*}
   $ m\left((\lambda_\mu,\lambda_\sigma),x\right)
    = 
    % \frac{1}{\sqrt{13}}
    %\begin{pmatrix}
       ( \lambda_\mu - x,
        \lambda_{\mu}^2 + \lambda_\sigma^2 - x^2)'$
   % \end{pmatrix}
%\end{equation*}
% \begin{equation}
% 	s\left((\lambda_\mu,\lambda_\sigma),x\right)
%     	= \frac{1}{4}\left( \lambda_\mu - x \right)^2 + 
%     	\frac{1}{4}\left( \lambda_\sigma^2 +  \lambda_\mu^2 - x^2\right)^2
%     	\;,
% \end{equation}
which satisfies the uniform boundedness conditions required in order to generate a family of test supermartingales according to Lemma~\ref{lem:test-smg-family-bded-ident}. Using this family of tests, we apply the~\ref{eq:ftl-definition} algorithm. The result of this test and of a confidence set on a single simulated path of the data generating process is displayed in Figure~\ref{fig:mean-sd-iid-beta}.

\begin{figure}[ht]
	\centering
	\begin{subfigure}[b]{0.495\textwidth}
		\centering
        	\includegraphics[width=\textwidth]{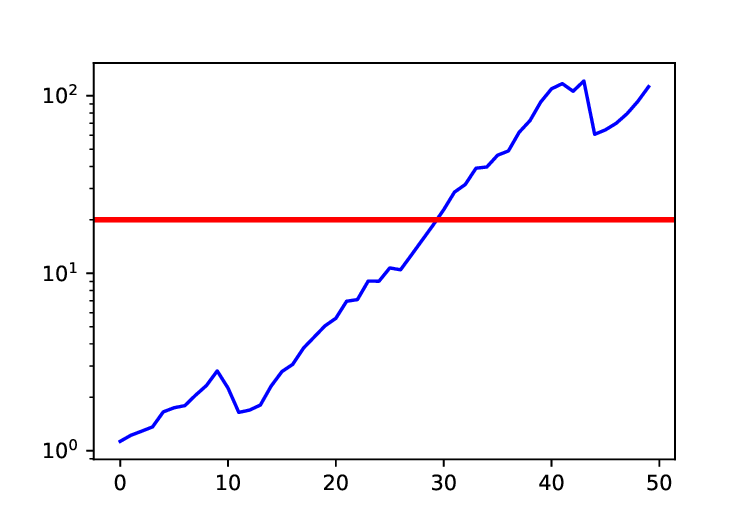}
        \caption{Test supermartingale and rejection threshold.}
        \label{subfig:mean-sd-iid-fig-test}
	\end{subfigure}
	\begin{subfigure}[b]{0.495\textwidth}
		\centering
        	\includegraphics[width=\textwidth]{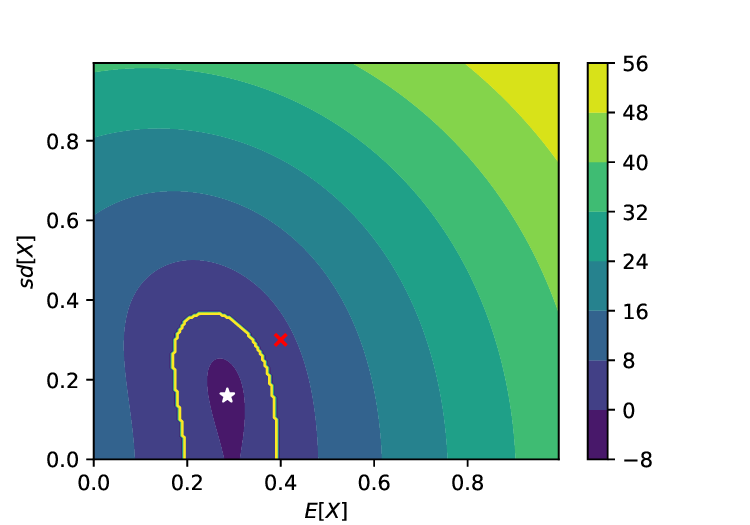}
        \caption{95\% confidence set at $t=50$.}
        \label{subfig:mean-sd-iid-fig-confint}
	\end{subfigure}
	\caption{
	Hypothesis tests and confidence sets %for the functional~\eqref{eq:mean-sd-functional} 
	for the joint mean and standard deviation of the i.i.d. Beta data generating process. 
	\textbf{(\ref{subfig:mean-sd-iid-fig-test})} The path of the test supermartingale $W_t$ constructed to test the hypothesis~\eqref{eq:mean-sd-hypothesis} (blue) with $\alpha=0.05$ rejection threshold (red). 
	\textbf{(\ref{subfig:mean-sd-iid-fig-confint})} The anytime-valid confidence set obtained at $t=50$ (yellow line). The color gradient represents the value of $\log(W_{50})$ over the space of possible of null hypotheses. The white $\star$ represents the true parameter values, and the red x represents the null hypothesis that is being tested in Figure (1a).
	}
	\label{fig:mean-sd-iid-beta}
\end{figure}

We conduct a second experiment with the same data generating distribution on the elicitable functional
%\begin{equation} \label{eq:var-cvar-functional}
    $T(Q) = %\begin{pmatrix}
		(\text{VaR}_{0.05}(Q),
		\text{CVaR}_{0.05}(Q))'$,
%	\end{pmatrix}
%	\;,
%\end{equation}
producing both the Value-at-Risk, that is, the 5\%-quantile, and conditional Value-at-Risk or expected shortfall, that is $\E_Q[X_i \mid X_i < \text{VaR}_{0.05}(Q) ]$, where this functional is identifiable with identification function 
%\[
$m((\lambda_v,\lambda_c),x) = %\begin{pmatrix}[1.5]
   (\mathds{1}_{x \le \lambda_v} - \alpha_0,
   x\mathds{1}_{x \le \lambda_v} - \alpha_0 \lambda_c)'
%\end{pmatrix} \;,
$
where $\alpha_0=0.05$, see \citet[Theorem 5.2]{FisslerZiegel2016}.
We test the hypothesis
\begin{equation} \label{eq:var-cvar-hypothesis}
    \Hnull = \left\{ P \in \Mcal_1(\Xcal^\N)\colon T(P(X_t \in \cdot \mid \Fcal_{t-1})) = 
    \begin{pmatrix} 0.2 \\ 0.1 \end{pmatrix} \text{ for all } t \in \N, \ P\text{-a.s.}
    \right\}
\end{equation}
where we approximately have that $T(P_\text{GT}) \approx \left( \begin{smallmatrix} 0.06 \\ 0.04 \end{smallmatrix} \right)$ under the true data generating measure $P_\text{GT}$. Since the underlying random variables are bounded, we construct a composite test supermartingale using the family of test supermartingales for uniformly bounded elicitable hypotheses given in Lemma~\ref{lem:test-smg-family-bded-elic}. We optimize this composite martingale using the~\ref{eq:ftl-definition} algorithm. 
% with regularization term $r_0(\lambda_\text{VaR},\lambda_\text{ES}) = \frac{1}{2}(\lambda_\text{VaR} - 0.3 )^2 + \frac{1}{2}(\lambda_\text{ES} - 0.3 )^2$ and $r_i \equiv 0$ for all $i\geq 1$, where we regularize towards the null hypothesis using a quadratic function. The results for this example are provided in Figure~\ref{fig:var-cvar-iid-beta}.

\begin{figure}[ht]
	\centering
	\begin{subfigure}[b]{0.495\textwidth}
		\centering
        	\includegraphics[width=\textwidth]{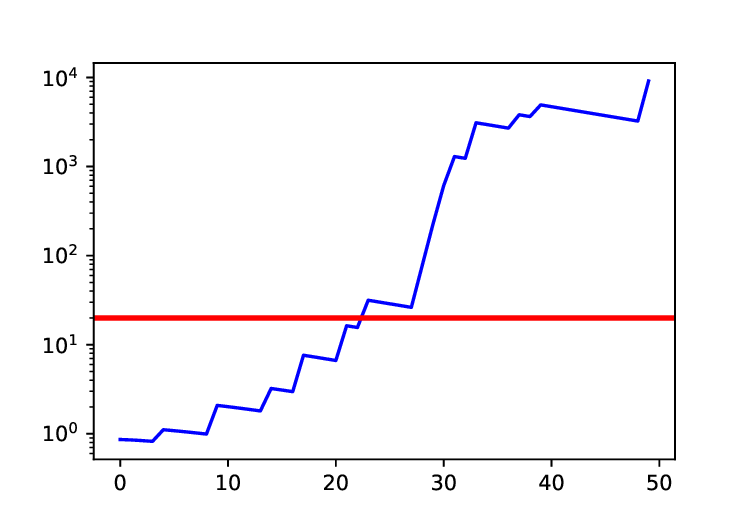}
        \caption{Test supermartingale and rejection threshold.}
        \label{subfig:var-cvar-iid-fig-test}
	\end{subfigure}
	\begin{subfigure}[b]{0.495\textwidth}
		\centering
        	\includegraphics[width=\textwidth]{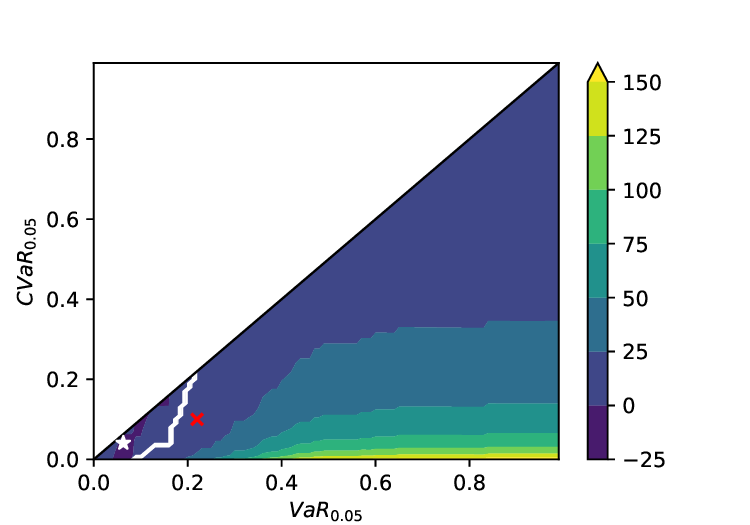}
        \caption{95\% confidence set at $t=150$}
        \label{subfig:var-cvar-iid-fig-confint}
	\end{subfigure}
	\caption{
	Hypothesis tests and confidence sets %for the functional~\eqref{eq:var-cvar-functional} 
	for the joint quantile and expected shortfall of the i.i.d. Beta data generating process. 
	\textbf{(\ref{subfig:var-cvar-iid-fig-test})} The path of the test supermartingale $W_t$ constructed to test the hypothesis~\eqref{eq:var-cvar-hypothesis} (blue) with $\alpha=0.05$ rejection threshold (red). 
	\textbf{(\ref{subfig:var-cvar-iid-fig-confint})} The anytime-valid confidence set obtained at $t=50$ (white line). The color gradient represents the value of $\log(W_{50})$ over the space of possible of null hypotheses, lying in the space of values such that $	\text{CVaR}_{0.05} \leq \text{VaR}_{0.05}$. The white $\star$ represents the true parameter values, and the red x represents the null hypothesis that is being tested in Figure (2a).
	}
	\label{fig:var-cvar-iid-beta}
\end{figure}

%\subsection{AR(1) Process}

Let us now consider the problem of estimating the linear coefficient of an AR(1) model. Specifically, assume that the data generating process is
\begin{equation*}
    X_{t+1} =  \beta X_t + \xi_{t+1} \,
\end{equation*}
where the $\{\xi_t\}_{t\in\N}$ are a 1-sub-Gaussian martingale difference sequence. This assumption may be replaced with $\sigma$-sub-Gaussianity for any $\sigma>0$ with straightforward minor modifications. We are interested in estimating and testing hypotheses regarding the value of $\beta\in\R$. For this purpose, we write this coefficient as the functional
%\begin{equation*}
 $   \beta(P(\cdot\mid\Fcal_t))
    = \min_{\beta^\prime \in \R} \E_P[ (X_{t+1} - \beta^\prime X_t)^2 \mid \Fcal_t ]$,
%    \;,
%\end{equation*}
which admits the identification function $m(\beta_0,(x_{t},x_{t+1}))= (x_t/|x_t|) ( \beta_0 x_t - x_{t+1})$ which will be $1$-sub-Gaussian with $\psi(u) = \frac{1}{2} u^2$ as shown in Example~\ref{ex:sub-psi-ark}. Hence, we may define the family of test supermartingales,
\begin{equation*}
    L_t^{\eta} = \prod_{i=1}^t \exp\left( \eta \, m(\beta_0,(x_{i},x_{i+1})) - \frac{\eta^2}{2}   \right)
    \;,
\end{equation*}
indexed by $\eta \in \Theta = \R$. Using this family of test supermartingales and noting that $\eta \mapsto \log L_t^{\eta}$ is strongly concave, we construct a composite test supermartingale by applying the the \ref{eq:ftl-definition} algorithm. This update rule has the advantage that it can be computed in closed form at each iteration. Indeed, letting $\overline{m}_t = (1/t) \sum_{i=1}^t m(\beta_0,(x_{i},x_{i+1}))$, we have the update rule
%\begin{equation*}
 $   \eta_{t+1} = \argmin_{\eta\in \Theta} \left\{  -\log L_t^\eta \right\}
    =
    \argmax_{\eta\in\R} \left\{  \eta \, \overline{m}_t - \frac{1}{2} \eta^2 \right\}
    =
    \overline{m}_t$.
%    \;.
%\end{equation*}
Applying this OCO algorithm to randomly generated data from the $AR(1)$ process, we test the hypothesis that $\Hcal_0 = \{ P \in \Mcal_1(\Xcal^\N) \,:\, \beta(P(\cdot\mid\Fcal_t)) = 0.65 \; \text{for all } t\in\N \}$ when data is generated from an $AR(1)$ process with $\beta=0.5$ and where $\{\epsilon_t\}_{t\in\N}$ are i.i.d. with $\epsilon_t\sim \Ncal(0,0.8)$. The results of this hypothesis test are displayed in Figure~\ref{subfig:ar1-test}. Similarly, in Figure~\ref{subfig:ar1-confint}, we display the running anytime-valid confidence sequence generated using the method described in Section~\ref{sec:confidence-sequences}.

\begin{figure}[th]
	\centering
	\begin{subfigure}[b]{0.495\textwidth}
		\centering
        	\includegraphics[width=\textwidth]{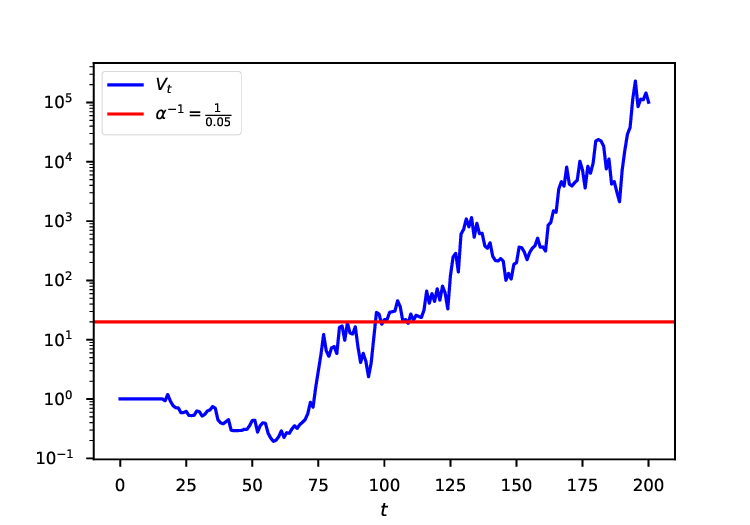}
        \caption{Testing }
        \label{subfig:ar1-test}
	\end{subfigure}
	\begin{subfigure}[b]{0.495\textwidth}
		\centering
        	\includegraphics[width=\textwidth]{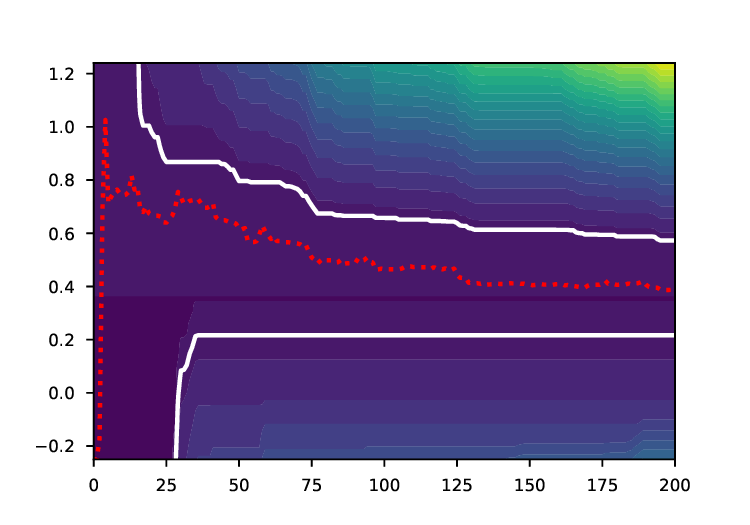}
        \caption{Confidence Sequence}
        \label{subfig:ar1-confint}
	\end{subfigure}
	\caption{Hypothesis tests and confidence sets for the parameter $\beta$ of an AR(1) data generating process. 
	\textbf{(\ref{subfig:ar1-test})} The path of the test supermartingale $W_t$ (blue) constructed and the rejection threshold (red), constructed to test the hypothesis that $\beta(P(\cdot\mid\Fcal_t)) = 0.65$ at the level $\alpha=0.05$. 
	\textbf{(\ref{subfig:ar1-confint})} The boundary of the anytime-valid confidence sequence (white) and the running estimate of the parameter $\beta$ (red). The color gradient represents the value of $\log(W_{t})$ over the space of possible of null hypotheses ($y$ axis) at each point in time.
}
	\label{fig:ar1-test-confint}
\end{figure}

%\section{Discussion} \label{sec:conclusion}

%\printbibliography
\bibliography{bib-files/seqstat.bib}

\appendix

\section{One-Sided and Set-Valued Hypotheses}
\label{sec:one-sided-set-valued-hypotheses}

The methodology considered in this paper may be extended towards testing one-sided and more generally, set-valued hypotheses on elicitable and identifiable functionals. Indeed, for a set $K\subseteq\Lambda$ let us consider the hypothesis 
\[
\Hnull(K) = \left\{ P \in \Mcal_1(\Xcal^\N)\colon T( P( X_t \in \cdot \mid \Fcal_{t-1})) \subset K \text{ for all } t \in \N, \text{ $P$-a.s.} \right\}.
\]
where $K$ can be thought of as a set or interval in which we believe the functional $T$ should lie when applied to the data-generating measure.

We can test the hypothesis $\Hnull(K)$ by building an e-process out of test supermartingales for pointwise hypotheses. Indeed, for each $\lambda_0\in K$, %define the pointwise hypotheses
%\[
%    \Hnull^{\lambda_0} = 
%    \left\{ P \in \Mcal_1(\Xcal^\N)\colon T( P( X_t \in \cdot \mid \Fcal_{t-1})) = \lambda_0 \text{ for all } t \in \N, \text{ $P$-a.s.} \right\}
%\]
%and 
assume that $(W^{\lambda_0}_t)_{t\in\N}$ is a test supermartingale for $\Hnull(\{\lambda_0\})=\Hnull(\lambda_0)$ given at \eqref{eq_null_seq_T}. If we define the composite process $(M_t^K)_{t\in\N}$ as $M_t^K = \min_{\lambda_0 \in K} W_t^{\lambda_0}$, it is easy to see that for any $\Fcal$-adapted stopping stopping time $\tau$, we have that
\begin{equation*}
    \E\left[ M_\tau^K \right]
    = \E\left[ \min_{\lambda_0 \in K} W_\tau^{\lambda_0} \right]
    \leq 
    \min_{\lambda_0 \in K} \E\left[ W_t^{\lambda_0} \right]
    \leq 1
    \;,
\end{equation*}
showing that $M_\tau^K > \alpha^{-1}$ can serve as a valid test for $\Hnull^K$. In principle, any test derived over the course of this paper for pointwise hypotheses can be used to generate the set of test supermartingales $W_t^{\lambda_0}$ used to test the set-valued hypothesis.

\section{Proofs}

\subsection{Proof of Lemma \ref{lem:predictably-controlled-test-smg}}
\begin{proof}
Since $W$ is nonnegative and $W_0 = 1$, we only need to check the $P$-supermartingale property for any $P \in \Hnull$. For each $t \in \N$ we use Tonelli's theorem, the fact that $W_t$ and $\pi_{t+1}$ are $\Fcal_t$-measurable, the $P$-supermartingale property of $L^\theta$, and the fact that $\pi_{t+1}$ has total mass one, to obtain
	\begin{align*}
		\E_P\left[ W_{t+1} \mid \Fcal_t \right] 
		&=
		W_t \, \E_P\left[ 
		\int_{\theta\in\Theta} 
			\frac{L_{t+1}^\theta \;}{L_{t}^\theta} \, 
			\pi_{t+1}(d\theta)  
		\mid \Fcal_t \right]
		\\&=
		W_t \, 
		\int_{\theta\in\Theta} 
		\E_P\left[ 
		\frac{L_{t+1}^\theta \;}{L_{t}^\theta} \, 
		\mid \Fcal_t \right]
		\pi_{t+1}(d\theta)
		\leq 
		W_t \, \pi_{t+1}(\Theta) = W_t.
	\end{align*}
This shows that $W$ is a $P$-supermartingale.

\subsection{Proof of Lemma \ref{lem:test-smg-family-bded-elic}}
	First, we note that $\inf_{\lambda\in\Lambda, x\in\Xcal} \{ s(\lambda_0,x) - s(\lambda^\prime,x) \} > -1$ implies that 
	%$\left( 1 + s(\lambda_0,x_i) - s(\lambda,x_i) \right) > 0$ almost surely, and hence 
	each $L^\lambda$ is nonnegative. Next, for each $t\in\N$, $\lambda\in\Lambda$ and $P\in\Hnull$, we may compute
	\begin{align*}
		\E_P\left[ L_{t+1}^\lambda \mid \Fcal_t \right]
		= L_t^\lambda \left(1 + \E_P\left[  s(\lambda_0,X_{t+1}) - s(\lambda^\prime,X_{t+1})  \mid \Fcal_t \right]\right)
		\;.
	\end{align*}
	By assumption, since $P\in\Hnull$ with $\Hnull$ defined according to equation~\eqref{eq_elicitable_null}, we have that \[-1 < \E_P\left[  s(\lambda_0,X_{t+1}) - s(\lambda^\prime,X_{t+1})  \mid \Fcal_t \right] \leq 0\;,\] which implies that
	\begin{equation*}
		L^\lambda_t \left( 1 + \E_P\left[  s(\lambda_0,X_{t+1}) - s(\lambda^\prime,X_{t+1})  \mid \Fcal_t \right] \right) \leq L^\lambda_t
		\;,
	\end{equation*}
	and hence $\E_P\left[ L_{t+1}^\lambda \mid \Fcal_t \right] \leq L^\lambda_t$. Lastly, 
	%since $L^\lambda_t > 0$ and 
	$L^\lambda_0 = 1$.
	%, we conclude that each $L^\lambda$ is an $\Hnull$ test supermartingale.

\end{proof}
\subsection{Proof of Lemma \ref{lem:test-smg-family-bded-ident}}
\begin{proof}
    We begin with the proof that $\Abd_{m,\lambda_0}$ is convex with non-empty interior. By assumption, there exists $C>0$ such that $\sup_{x\in\Xcal}\| m(\lambda_0,x) \| \leq C < \infty$. Hence, for any $\|\eta\| <  1/C$, we have by the Cauchy-Schwarz inequality that
    \begin{equation*}
        \langle \eta , m(\lambda_0,x) \rangle
        \geq 
        - \| \eta \| \| m(\lambda_0,x) \| 
        \geq - \| \eta \| \sup_{x\in\Xcal}\| m(\lambda_0,x) \| 
        > -1
        \;,
    \end{equation*}
    showing that $\Abd_{m,\lambda_0}$ has non-empty interior. Convexity follows from linearity of the scalar product in the first argument.
    %To demonstrate convexity, we note that for any $\rho\in[0,1]$ and $\eta_0,\eta_1\in \Abd_{m,\lambda_0}$,
    %\begin{align*}
    %    \inf_{x\in\Xcal} \left\langle
    %    \rho \,\eta_0 + (1-\rho)\,\eta_1
    %    \mathrel{,} m(\lambda_0,x)
    %    \right\rangle
    %    &\geq
    %    \rho \inf_{x\in\Xcal} \left\langle
    %    \eta_0
    %    \mathrel{,} m(\lambda_0,x)
    %    \right\rangle
    %    +
    %    (1-\rho) \inf_{x\in\Xcal} \left\langle
    %    \eta_1
    %    \mathrel{,} m(\lambda_0,x)
    %    \right\rangle
    %    \\ &>
    %    -\rho - (1-\rho)
    %    = -1
    %    \;,
    %\end{align*}
    %and hence $\rho \,\eta_0 + (1-\rho)\,\eta_1 \in \Abd_{m,\lambda_0}$.
    
	%5Now we proceed to show that the $L_t^\eta$ are valid test martingales for all $\eta\in %\Abd_{m,\lambda_0}$ We first note that 
	%By the definition of $\Abd_{m,\lambda_0}$, 
	For $\eta \in \Abd_{m,\lambda_0}$, we have
	%\begin{align*}
		$1 + \langle \eta \mathrel{,} m(\lambda_0,x_i)\rangle> 0$, 
%		\;,
%	\end{align*}
	and hence the $L^\eta$ are all nonnegative. For each $t\in\N$, and $P\in\Hnull$, it holds that
	\begin{align*}
		\E_P\left[ L_{t+1}^\eta \mid \Fcal_t \right]
		= L_t^\eta \left(1 + \langle \eta \mathrel{,} \E_P\left[ m(\lambda_0,x_i) \mid \Fcal_t \right] \rangle \right)
		\;.
	\end{align*}
	By equation~\eqref{eq_identifiable_null}, we have that $\E_P\left[ \langle \eta \mathrel{,} m(\lambda_0,X_i)\rangle  \mid \Fcal_t \right] = 0$ and thus $\E_P\left[ L_{t+1}^\eta \mid \Fcal_t \right] = L_t^\eta$ for all $t$. Lastly, since $L_0^\eta = 1$ and $L_t^\eta >0$ we conclude that $L^\eta$ is an $\Hnull$ test martingale.
\end{proof}

\subsection{Proof of Lemmas \ref{lem:test-smg-family-subpsi-elic} and \ref{lem:test-smg-family-subpsi-ident}}
% \fbox{check if this needs adjusting}
    \begin{proof}
    Let us first consider the setting of Lemma~\ref{lem:test-smg-family-subpsi-elic}.
	By construction, each of the processes $L^{\lambda,u}$ is nonnegative and satisfies $L^{\lambda,u}_0=1$. Fixing $(\lambda,u)$, we have that for all $P\in\Hnull$,
	\begin{align*}
		\E_P\left[ 
			L^{\lambda,u}_{t+1} \mid \Fcal_{t}
		\right] 
		= 
		\E_P\left[ 
			e^{u Y_{t+1}^\lambda - V_{t+1}^\lambda \psi_\lambda(u) } \mid \Fcal_{t}
		\right] 
		\leq e^{u Y_t^\lambda - V_t^\lambda \psi_\lambda(u) }
		= L^{\lambda,u}_{t} \text{ a.s.},
	\end{align*}
	which follows by noting that $Y^\lambda$ is a supermartingale according to the definition of $\Hcal_0$ and by applying Lemma~\ref{lem:sub-psi-iff-neg-mean}. Hence, each $L^{\lambda,u}$ is a valid test supermartingale.

    For Lemma~\ref{lem:test-smg-family-subpsi-ident}, we note that the proof works analogously to that of Lemma~\ref{lem:test-smg-family-subpsi-elic}.
    \end{proof}

\subsection{Proof of Lemma \ref{lem:concavity-test-smg}}

\begin{proof}
	We prove this claim by proving it for the four separate classes of test supermartingales of Lemmas~\ref{lem:test-smg-family-bded-elic},~\ref{lem:test-smg-family-bded-ident},~\ref{lem:test-smg-family-subpsi-elic} and~\ref{lem:test-smg-family-subpsi-ident}.

	\noindent\emph{Lemma~\ref{lem:test-smg-family-bded-elic}.} In this case, the family of test supermartingales $\{L^\theta\}_{\theta\in\Theta}$, where $\theta=\lambda\in\Lambda$, can be written as
	\begin{equation}
		\log L^\theta_t
		= \sum_{1\leq i \leq t} \log\left( 1 + 
			s(\lambda_0,x_i) - s(\lambda,x_i)
		\right)
		\;.
	\end{equation}
Let us define $f_i(\lambda) = s(\lambda_0,x_i) - s(\lambda,x_i)$, which is concave since $s(\lambda,x_i)$ is convex. Using the concavity and monotonicity of $\log$ and the concavity of the $f_i$, we have that for any $\lambda,\lambda^\prime \in \Lambda$ and $\rho\in(0,1)$,
	\begin{align*}
		\log( 1 + f_i( \rho \lambda + (1-\rho) \lambda^\prime ) )
		&\geq \log( \rho( 1 + f_i(\lambda ) ) + (1-\rho) f_i(\lambda^\prime) )
		\\ &\geq \rho \, \log( 1 + f_i(\lambda ) ) + (1-\rho) \log(1+f_i(\lambda^\prime) )
		\;,
	\end{align*}
	showing that $\log( 1 + f_i(\lambda ) )$ is concave.
	Hence, $\log L^\theta_t$ is a sum of concave functions of $\theta$, and hence is concave as well.

	\noindent\emph{Lemma~\ref{lem:test-smg-family-bded-ident}.} In this case, the family of test supermartingales $\{L^\theta\}_{\theta\in\Theta}$, where $\eta = \theta \in \Theta \subseteq \Abd_{m,\lambda_0}$, can be written as
	\begin{equation}
		\log L^\theta_t
		= \sum_{1\leq i \leq t} \log\left( 1 + \left\langle \eta \mathrel{,} m(\lambda_0,x_i) \right\rangle\right)
		\;.
	\end{equation}
	By the concavity of the $\log$, each of the $\log\left( 1 + \left\langle \eta \mathrel{,} m(\lambda_0,x_i) \right\rangle\right)$ are concave in $\eta$. Hence, $\log L^\theta_t$ is a sum of concave functions of $\theta$, and hence is concave as well.

	\noindent\emph{Lemma~\ref{lem:test-smg-family-subpsi-elic}.} In this case, the family of test supermartingales $\{L^\theta\}_{\theta\in\Theta}$, where $(\lambda,u) = \theta \in \Theta \subseteq [0,u_\text{max})\times \Lambda$, can be written as
	\begin{equation}
		\log L^\theta_t
		= \sum_{1\leq i \leq t}
		\left\{
		u(\theta) \, \left( s(\lambda_0,x_i) - s(\lambda(\theta),x_i) \right)
		- v_i^{\lambda(\theta)} \psi_{\lambda(\theta)}(u(\theta))
		\right\}
		\;.
	\end{equation}
	By assumption, each summand is concave in $\theta=(\lambda,u)$ and hence $\log L^\theta_t$ must also be concave in $\theta$.

	\noindent\emph{Lemma~\ref{lem:test-smg-family-subpsi-ident}.} In this case, the family of test supermartingales $\{L^\theta\}_{\theta\in\Theta}$, where $(\eta,u) = \theta \in \Theta \subseteq \Apsi_{m,\lambda_0} \times [0,\umax)$, can be written as
	\begin{equation}
		\log L^\theta_t
		= \sum_{1\leq i \leq t}
		\left\{
		\left\langle u(\theta) \mathrel{,} m(\lambda_0,x_i) \right\rangle
		- v_i^{\lambda(\theta)} \psi_{\eta(\theta)}(u(\theta))
		\right\}
		\;.
	\end{equation}
	By assumption, each summand is concave and hence $\log L_t^\theta$ is a sum of concave functions and hence is itself almost surely concave.

\end{proof}

\end{document}